\documentclass[11pt,a4paper]{amsart}
\usepackage[T1]{fontenc}
\usepackage[applemac]{inputenc}
\usepackage{amsmath,amssymb,amsthm}
\usepackage{mathtools}
\usepackage{xcolor}
\usepackage{hyperref}
\hypersetup{colorlinks=true,linkcolor=blue!60!black,citecolor=blue!60!black,urlcolor=blue!60!black,pdftitle={Obstructions to the ampleness of invariant jet differentials on complete intersections},pdfauthor={Simone Diverio}}

\theoremstyle{plain}
\newtheorem{theorem}{Theorem}[section]
\newtheorem{proposition}[theorem]{Proposition}
\newtheorem{lemma}[theorem]{Lemma}
\newtheorem{corollary}[theorem]{Corollary}
\newtheorem{conjecture}[theorem]{Conjecture}
\theoremstyle{definition}

\newtheorem{example}[theorem]{Example}
\theoremstyle{remark}
\newtheorem{remark}[theorem]{Remark}
\newtheorem{convention}[theorem]{Convention}

\newcommand{\PP}{\mathbb{P}}
\newcommand{\CC}{\mathbb{C}}
\newcommand{\ZZ}{\mathbb{Z}}
\newcommand{\NN}{\mathbb{N}}
\newcommand{\cO}{\mathcal{O}}
\newcommand{\Sym}{\operatorname{Sym}}
\newcommand{\Gr}{\operatorname{Gr}}
\newcommand{\Om}{\Omega^1_X}
\newcommand{\OmZ}{\Omega^1_Z}
\newcommand{\E}[2]{E_{#1,#2}T^*_X}
\newcommand{\Ez}[2]{E_{#1,#2}T^*_Z}
\newcommand{\Sch}[1]{\Gamma^{#1}}
\DeclareMathOperator{\GL}{GL}
\DeclareMathOperator{\SL}{SL}
\newcommand{\Diff}{\mathcal{D}\mathrm{iff}}

\title[Obstructions to the ampleness of invariant jet differentials]{Obstructions to the ampleness of invariant jet differentials\\ on complete intersections}

\author{Simone Diverio}
\address{Dipartimento di Matematica Guido Castelnuovo, Sapienza Universit\`a di Roma, Piazzale Aldo Moro 5, I-00185 Roma, Italy}
\email{simone.diverio@uniroma1.it}

\thanks{S.~Diverio is partially supported by the Sapienza Universit\`a di Roma research project ``Algebro-geometric, analytic, and topological aspects of real and complex varieties'' (ALGEANT), Progetti di Ricerca di Ateneo 2025, CUP B83C26001140005.}
\date{\today}
\makeatletter
\@namedef{subjclassname@2020}{\textup{2020} Mathematics Subject Classification}
\makeatother
\keywords{Invariant jet differentials, Demailly--Semple jet bundles, ample vector bundles, complete intersections, Schur powers of the cotangent bundle, Br\"uckmann--Rackwitz vanishing theorem, Diverio--Trapani conjecture}

\begin{document}

\begin{abstract}
We introduce the analogue, for invariant jet differentials, of the principal symbol of a differential operator: the component of maximal degree in the highest derivative. The top symbol identifies a canonical quotient bundle of the Demailly--Semple bundle $E_{k,m}T^*_X$, with values in jet differentials of lower order, and together with the Br\"uckmann--Rackwitz vanishing theorem it gives a non-ampleness criterion for $E_{k,m}T^*_Z$ on complete intersections $Z\subset\PP^N$. On the Demailly--Semple tower the top symbol is the restriction to a thickening of the vertical divisor, inside the singular jets.

The criterion applies in orders $3$ and $4$: $E_{3,5q}T^*_X$ is not ample on any smooth hypersurface $X\subset\PP^4$, and $E_{4,m}T^*_X$ is not ample for $m\ge7$ with $m\equiv0,1,2\pmod 7$ on any smooth hypersurface $X\subset\PP^5$ of degree at least $2$. This contradicts the Diverio--Trapani conjecture at its own threshold.
\end{abstract}

\maketitle

\section{Genesis}

The year was 2010. Together with Stefano Trapani, we had the audacity to state a conjecture. For over a decade and a half, it sat there quietly---aside from a brilliant paper by Ya Deng in the Annales de l'\'Ecole Normale Sup\'erieure that briefly tricked us into believing we were mathematical visionaries.

Fast forward to the summer of 2026: AI models were demolishing legendary mathematical problems left and right. Naturally, inside the WhatsApp group of our Algebra and Geometry research group at Sapienza---a chat composed of supposedly mature academics behaving like high school pranksters---a running joke was born. While silicon giants were tackling humanity's hardest open problems, the mighty Diverio--Trapani Conjecture stood invincible. We liked to pretend Sam Altman was losing sleep over us.

Then my colleague and friend Giulio Tiozzo---working in dynamical systems and quite foreign to complex algebraic geometry---armed with a ChatGPT Pro account and zero respect for my ego, fed the conjecture to the machine. It took the AI precisely 7 minutes and 13 seconds to spit out a valid counterexample. Sixteen years of pride, completely dismantled in less time than it takes to whisk together a batch of pancake batter.

Once the laughter at my mathematical downfall settled, I actually examined what the algorithm had produced. To my annoyance, the counterexample wasn't just correct---it was genuinely interesting. Since Giulio's mathematical labour on this paper effectively peaked the second he hit `Enter', I stepped in to do what was left for a human: checking every step against the literature, stripping away the digital fluff, generalising the result, and taking responsibility for it. 

This paper is the outcome of that humbling episode: one person provided the prompt, one chatbot provided the proof, and the other person spent a few weeks understanding it, making sure it was actually true, generalising it to the best of his knowledge---from hypersurfaces of $\PP^4$ to complete intersections of any dimension, and from order three to order four---and, with a little more ambition, looking for a unified framework in which the counterexamples found by the machine would appear as instances of a deeper phenomenon: the principal symbol of a jet differential, to which Section~\ref{sec:symbol} is devoted.

\subsection*{Who did what}

To be precise about who did what: ChatGPT-6 Astra produced the original argument for hypersurfaces of $\PP^4$, namely the identification of the extremal quotient $\Gamma^{(2q,q,0)}\Omega^1_X$ of $E_{3,5q}T^*_X$ from Rousseau's generators and the contradiction with Br\"uckmann--Rackwitz vanishing, that is, Sections~\ref{sec:quotient}--\ref{sec:nonample} below in a rougher form.

The following came later and was done by me. The main conceptual addition is the reading of the extremal quotient as the \emph{principal symbol} of a jet differential with respect to its highest derivative, in the precise sense of Section~\ref{sec:symbol}. The original argument already called the top homogeneous component a ``symbol'', and checked by hand that in the case at hand it transforms as a tensor; what is not there is the general structure behind this: the analogy with the principal symbol of a differential operator, the description of the symbol of an arbitrary invariant jet differential as an invariant jet differential of lower order (Proposition~\ref{prop:symbol}, whose non-invariant counterpart is a formula in Demailly's notes), the explanation of when it is a tensor and of how many rows its irreducible summands can have, and the general non-ampleness criterion which follows (Theorem~\ref{thm:criterion}). It is through this analogy that I understood why the construction works, and it is the framework in which all the results of this note are stated.

Besides the symbol, the following are also mine and not the machine's: drawing the consequences in every dimension of the fact, already remarked by Rousseau, that his generators describe the invariant $3$-jet algebra in every dimension, hence Theorem~\ref{thm:ci}; working out the order-four case from Merker's list of bi-invariants, hence Theorem~\ref{thm:k4}; observing what the multiplication maps $\Sym^lE_{k,m}\to E_{k,ml}$ do and do not give, hence Corollary~\ref{cor:divisible}; and writing Remark~\ref{rem:varym}.

\subsection*{Authorship}

An important (and serious) note on authorship is due. Initially, caught up in the enthusiasm of the moment, we had even playfully considered listing ChatGPT-6 Astra alongside Giulio and myself as official co-authors---a lighthearted stunt intended, much like the tone of this Genesis, as a sarcastic commentary on the current academic climate and the modern rush toward AI-assisted publishing. While dropping the machine from the byline was an obvious given, Giulio's co-authorship prompted a far more genuine and thoughtful ethical discussion. After talking it over and reflecting deeper on the matter, we mutually agreed on this final configuration.

We realised that in this emerging era of artificial intelligence, mathe\-ma\-ti\-cians---especially mature ones---must approach these tools with a clear ethical framework. Since his contribution was fundamentally the initial prompt---however brilliantly disruptive it proved to be---and the actual mathematical digestion, verification, generalisation, and ultimate responsibility fell entirely on my shoulders, we felt it inappropriate for him to claim formal co-authorship. This decision underscores a crucial point: while machines can dismantle decades-old conjectures in minutes, the human responsibility of understanding, validating, and taking genuine ownership of the mathematics remains paramount.

\subsection*{A note added to the second version}

I am well aware that the self-mocking and sarcastic tone of this Genesis may look out of place in a research paper. Tone apart, I believe that the story deserves to be told, because it is, in its small way, typical of the times we are living in; this is why I decided to tell it.

Having said this, I find it just as important to stress the positive side of the interaction with the machine. In the first version of this note I had understood the construction technically and, knowing the subject, I had been able to extend it somewhat, to all dimensions and to order four, by means of the results of Rousseau and Merker. 

What is new in this second version is the principal symbol of a jet differential of Section~\ref{sec:symbol}, that is, an understanding of the mechanism behind the counterexamples; it is nowhere in the output of the machine, and it came only after turning the counterexample over in my mind for some weeks. 

If ChatGPT-6 Astra had not produced the counterexample, this paper would not exist and I would not have understood this. The counterexample, as it came out of the machine, was correct, and as cold as a machine; what has grown around it, the symbol and the explanation it provides, is human work, and it has improved my understanding of a subject I have been thinking about for many years. I would not have got there without that cold start.

\section{Statement, conventions, and the conjecture}\label{sec:statement}

In this section we fix our conventions on ampleness and on jet differentials, we recall the conjecture of Diverio and Trapani in its two published forms and discuss the quantifier on the weight $m$, which turns out to be the crux of the matter, we put the problem in the perspective of Debarre's conjecture and of its solution, and we state the main results. The section ends with an explicit account of which idea is new and which ones are merely assembled from the literature.

\subsection{Conventions}
All varieties and vector bundles are defined over $\CC$. A vector bundle $E$ on a projective variety is \emph{ample} if the tautological line bundle $\cO_{\PP(E)}(1)$ on the projective bundle of one-dimensional quotients of $E$ is ample \cite{Har66}. We shall use the following standard facts \cite[\S2--\S3]{Har66}, \cite[Chapter~6]{Laz04}:
\begin{enumerate}
\item[(H1)] every quotient bundle of an ample vector bundle is ample;
\item[(H2)] if $E$ is ample, then $\Sym^rE$ is generated by global sections for all $r\gg0$;
\item[(H3)] a quotient bundle of a globally generated bundle is globally generated; in particular a nonzero globally generated bundle has a nonzero global section.
\end{enumerate}

For a partition $\lambda=(\lambda_1\ge\lambda_2\ge\dots)$ we write $\Sch{\lambda}$ for the corresponding Schur functor, $\lambda'$ for the conjugate partition (so that $\lambda'_j$ is the length of the $j$-th column of the Young diagram of $\lambda$), and we call \emph{rows} of $\lambda$ its nonzero parts; thus $\Sym^m=\Sch{(m)}$ has one row and $\Lambda^2=\Sch{(1,1)}$ has two.

Let $X$ be a smooth complex variety of dimension $n$. Following Demailly \cite[Def.~6.7]{Dem97}, the bundle $\E{k}{m}$ of \emph{invariant jet differentials} of order $k$ and weighted degree $m$ is the vector bundle whose fibre over a point consists of the polynomials $P$ in the $k$-jet variables $(f',f'',\dots,f^{(k)})$, homogeneous of weighted degree $m$ with respect to the weights $\deg f^{(j)}=j$, which satisfy
\[
P\big(j_k(f\circ\varphi)\big)=\varphi'(0)^m\,P(j_kf)
\]
for every germ of biholomorphism $\varphi$ of $(\CC,0)$. It is a locally free sheaf, whose transition functions are induced by the chain rule for changes of coordinates on $X$ (which we recall in \S\ref{sec:symbol}), and when $n\ge2$ it is the direct image $\pi_{k,0\,*}\cO_{P_kT_X}(m)$ of the tautological bundle on the Demailly--Semple tower \cite[Thm.~6.8~ii)]{Dem97}. If $\E{k}{m}$ is ample for some $k,m>0$, then $X$ is Kobayashi hyperbolic \cite[Cor.~7.10]{Dem97}.

\begin{convention}
Throughout we use \emph{ordinary derivatives} $f^{(j)}(0)$ as jet variables, not Taylor coefficients divided by factorials. This affects numerical coefficients (\textsl{e.g.}\ the coefficient $3$ in \eqref{eq:gens} and \eqref{eq:chainrule}) and is the convention of \cite{Rou06,Mer10}.
\end{convention}

\subsection{The conjecture}\label{sec:conjecture}
Together with Trapani, we proposed in \cite{DT10} the following generalisation of Debarre's conjecture on ample cotangent bundles.

\begin{conjecture}[{Diverio--Trapani, \cite[\S4]{DT10}}]\label{conj:DT}
Let $Z\subset\PP^N$ be the intersection of at least $N/(k+1)$ general hypersurfaces of sufficiently high degree. Then $\Ez{k}{m}$ is ample, and therefore $Z$ is Kobayashi hyperbolic.
\end{conjecture}

If $c$ denotes the codimension, the condition $c\ge N/(k+1)$ reads $k\ge N/c-1$. Deng states the conjecture as follows.

\begin{conjecture}[{Deng, \cite[Conj.~0.4]{Deng20}}]\label{conj:Deng}
Let $Z\subset\PP^N$ be the complete intersection of $c$ general hypersurfaces of sufficiently high degree. Then the invariant jet bundle $\Ez{k}{m}$ is ample provided that $k\ge N/c-1$ and $m\gg0$.
\end{conjecture}

The quantifier on $m$ matters for what follows. In \cite[\S4]{DT10} no quantifier is given at all. A few lines before the conjecture, \cite{DT10} recalls, after \cite[Cor.~7.10]{Dem97}, that ampleness of $\Ez{k}{m}$ \emph{for some} $k$ (and $m$) implies hyperbolicity, which suggests the weakest reading; on the other hand the analogy with Debarre's conjecture, where $E_{1,m}=\Sym^m\OmZ$ is ample for one $m$ if and only if it is ample for every $m$, suggests that we simply had in mind ``for every $m$'', without noticing that for $k\ge2$ the quantifier is no longer innocuous (see Remark~\ref{rem:varym} for why it is not).

Taken literally, ``for every $m\ge1$'' cannot hold in the very range the conjecture was designed for: $E_{k,1}T^*_Z=\OmZ$ and $E_{k,2}T^*_Z=\Sym^2\OmZ$ for every $k$ (a polynomial of weight $2$ is a quadratic form in $f'$ plus $\sum_ic_if_i''$, and the term $\varphi''\sum_ic_if_i'$ produced by reparametrisation forces $c=0$: the first invariant involving $f''$ is the Wronskian, of weight $3$), and these are not ample when $2c<N$ by Schneider's theorem \cite{Sch92} (see \S\ref{sec:perspective}); more generally $E_{k,m}=E_{k-1,m}$ for $m<k$, since $f^{(k)}$ has weight $k$.

So ``for every $m$'' can only mean ``for every sufficiently large $m$'', which is the reading (A) below and the one adopted by Deng. We distinguish three readings, in decreasing order of strength:
\begin{enumerate}
\item[(A)] $\Ez{k}{m}$ is ample for \emph{all} sufficiently large $m$ (the literal reading of Conjecture~\ref{conj:Deng}, and presumably the intended reading of Conjecture~\ref{conj:DT});
\item[(A$'$)] there is an integer $M\ge1$ such that $\Ez{k}{m}$ is ample for all sufficiently large multiples $m$ of $M$;
\item[(B)] $\Ez{k}{m}$ is ample for \emph{some} $m\ge1$.
\end{enumerate}
Reading (A) is the natural one from the point of view of the Demailly--Semple tower, where $E_{k,m}T^*_Z$ is the direct image of $\cO_{Z_k}(m)=\cO_{Z_k}(1)^{\otimes m}$ and the weight is a mere multiple; it is exactly the transfer of this line-bundle intuition to the vector bundles $E_{k,m}T^*_Z$ which is unjustified, as Remark~\ref{rem:varym} explains: the obstruction found in Theorems~\ref{thm:ci} and~\ref{thm:k4} is sensitive to $m$ modulo $2k-1$, and whether ampleness itself is (that is, whether some $E_{k,m}T^*_Z$ in the remaining residue classes is ample) is not decided by our results.

\subsection{Perspective: Debarre's conjecture and its generalisations}\label{sec:perspective}
Conjecture~\ref{conj:DT} is the higher-order version of a story which begins with symmetric differentials. For $k=1$ one has $E_{1,m}T^*_Z=\Sym^m\OmZ$, whose ampleness is that of $\OmZ$, and the conjecture reduces to the following.

\begin{conjecture}[{Debarre, \cite[Conj.~18]{Deb05}}]\label{conj:Deb}
The cotangent bundle of the intersection in $\PP^N$ of at least $N/2$ general hypersurfaces of sufficiently high degrees is ample.
\end{conjecture}

Debarre was led to this statement by his theorem that the analogous property holds in abelian varieties \cite[Thms.~7 and~8]{Deb05}, and by a question of Schneider. The condition $c\ge N/2$ is necessary: Schneider proved that a smooth subvariety of $\PP^N$ with ample cotangent bundle has dimension at most equal to its codimension \cite{Sch92}; for complete intersections this is the case $\lambda=(m)$ of Theorem~\ref{thm:BR}.

Partial results were obtained by Brotbek, for complete intersection surfaces \cite{Bro14} and, by an explicit construction with deformations of Fermat-type hypersurfaces, for $\operatorname{codim}Z\ge3\dim Z-2$ \cite[Thm.~C]{Bro16}; the conjecture was then proved independently by Brotbek--Darondeau, in any smooth projective variety \cite[Thm.~0.1]{BD18}, and by Xie \cite{Xie18}, with an explicit uniform bound on the degrees; Deng made the Brotbek--Darondeau proof effective and improved the bound \cite{Deng20}.

Etesse subsequently proved the natural generalisation to Schur powers: $\Sch{\lambda}\OmZ$ is ample for a general complete intersection of high multidegree as soon as $(1+\ell)c\ge N$, where $\ell$ is the number of rows of $\lambda$, and this threshold is optimal by Theorem~\ref{thm:BR} \cite[Main Theorem and \S1]{Ete21}.

Conjecture~\ref{conj:DT} was proposed in \cite[\S4]{DT10} as the jet analogue of Conjecture~\ref{conj:Deb}, with the threshold $k\ge N/c-1$ modelled on the vanishing theorem for invariant jet differentials of \cite{Div08}, in the same way as Debarre's threshold $c\ge N/2$ is modelled on Schneider's obstruction: in both cases the guess is that ampleness sets in as soon as the elementary vanishing theorem no longer forbids sections.

The present note shows that for jet bundles of order $k\ge3$ this guess is wrong when read as ampleness of the vector bundles $E_{k,m}T^*_Z$, at least for the weights $m$ divisible by $2k-1$, because the graded pieces of $E_{k,m}T^*_Z$ do not all have the same number of rows: the pieces which carry sections, and which Etesse's theorem makes ample, are the Wronskian-like pieces with many rows, whereas the extremal piece with respect to the degree in $f^{(k)}$, which is a quotient, has few rows and no sections. Nothing of the kind happens for $k=1$, where $\Sym^m\OmZ$ is irreducible, nor for $k=2$, where the extremal piece and the Wronskian coincide (Remark~\ref{rem:ordertwo}).

What survives of Conjecture~\ref{conj:DT} is therefore, on the one hand, the positivity on the Demailly--Semple tower away from singular jets, which is exactly Deng's theorem of almost $k$-jet ampleness for $k\ge N/c-1$ \cite[Thm.~C]{Deng20}, and on the other hand the ampleness of the individual Schur bundles with enough rows, which is Etesse's theorem.

Whether some $E_{k,m}T^*_Z$ with $k\ge3$ and $m$ outside the residue classes covered by Theorems~\ref{thm:ci} and~\ref{thm:k4} can be ample remains open (apart from $E_{3,8}T^*_Z$, which is not, Proposition~\ref{prop:eight}); the iterated extremal quotient of $E_{3,5q+1}T^*_Z$ computed in Remark~\ref{rem:5divides} (top degree in $f'''$, then in $f''$) has three rows, hence is ample for general $Z$ in the range of Etesse's theorem, so the present method cannot decide the question.

\subsection{Main results}
\begin{theorem}\label{thm:main}
Let $X\subset\PP^4_\CC$ be a smooth hypersurface of any degree $d\ge1$, and let $q\ge1$ be an integer. Then $\E{3}{5q}$ is not ample.
\end{theorem}

\begin{theorem}\label{thm:ci}
Let $Z\subset\PP^N_\CC$ be a smooth complete intersection of $c$ hypersurfaces of degrees $\ge2$, and assume $3c<N$. Then $\Ez{3}{5q}$ is not ample for every $q\ge1$.
\end{theorem}

\begin{theorem}\label{thm:k4}
Let $Z\subset\PP^N_\CC$ be a smooth complete intersection of $c$ hypersurfaces of degrees $\ge2$, and assume $4c<N$. Then $\Ez{4}{m}$ is not ample for every $m\ge7$ with $m\equiv0,1,2\pmod7$. In particular $\E{4}{m}$ is not ample for such $m$ on any smooth hypersurface $X\subset\PP^5$ of degree $\ge2$.
\end{theorem}

Since $N/c-1\le3$ if and only if $N\le4c$, Theorem~\ref{thm:ci} shows that readings (A) and (A$'$) fail at $k=3$ for every $(N,c)$ with $3c<N\le4c$, and likewise Theorem~\ref{thm:k4} shows that they fail at $k=4$ for every $(N,c)$ with $4c<N\le5c$ (Corollaries~\ref{cor:conj} and~\ref{cor:conj4}); the smallest instances are
\begin{gather*}
(4,1),\ (7,2),\ (8,2),\ (10,3),\ (11,3),\ (12,3)\qquad\text{for }k=3,\\
(5,1),\ (9,2),\ (10,2),\ (13,3),\ (14,3),\ (15,3)\qquad\text{for }k=4.
\end{gather*} The case $(5,1)$ is the first one in which the hypersurface version of the conjecture is contradicted at its own threshold $k=N-1$ beyond $\PP^4$.

Together, the two ranges cover every dimension: for every $n\ge3$ there is a codimension $c$ with $n/4\le c<n/2$ for which the conjecture fails at its own threshold order $\lceil n/c\rceil$ (Corollary~\ref{cor:alldim}). They say nothing about reading (B) at the remaining weights, and the restrictions $5\mid m$, resp.\ $m\equiv0,1,2\pmod7$, cannot be removed by the present method, as far as we have checked (Remark~\ref{rem:5divides}), with the single exception of the weight $m=8$, where an iterated symbol still isolates a two-row quotient (Proposition~\ref{prop:eight}); what can be said is that a weight $m$ at which reading (B) held would be one at which the multiplication map $\Sym^5E_{3,m}T^*_Z\to E_{3,5m}T^*_Z$, resp.\ $\Sym^7E_{4,m}T^*_Z\to E_{4,7m}T^*_Z$, is not surjective (Corollary~\ref{cor:divisible}).

All three theorems are instances of a single criterion (Theorem~\ref{thm:criterion}): the top symbol of the jet differentials, in the sense of Section~\ref{sec:symbol}, involves only $f',\dots,f^{(s)}$, with $s=1$ for order $3$ and $s=2$ for order $4$, and $(s+2)c<N$. Below the threshold of the conjecture, where $Z$ carries no invariant jet differentials at all, non-ampleness is immediate (Remark~\ref{rem:below}); the content of the theorems is that at the threshold, where global jet differentials abound for large degrees, all of them have vanishing top symbol: the Wronskian-type invariants which carry the global sections never reach the top of the filtration.

They are, so to speak, transversal to Deng's theorems \cite{Deng20}: Deng proves positivity on the Demailly--Semple tower away from the singular jets, and the obstruction found here is read off along the vertical divisor of the tower, which lies in the singular jets, where his notion of almost $k$-jet ampleness deliberately does not look (Proposition~\ref{prop:divisor} and Remarks~\ref{rem:Deng} and~\ref{rem:vertical}).

\subsection{Organisation}
Section~\ref{sec:symbol} explains the construction which underlies the whole note, namely the \emph{principal symbol} of a jet differential with respect to its highest derivative; it states its formal properties (Proposition~\ref{prop:symbol}), derives from them a general non-ampleness criterion (Theorem~\ref{thm:criterion}), of which the main theorems are instances, and interprets the top symbol on the Demailly--Semple tower as the restriction of the jet differentials to a thickening of the vertical divisor, which lies in the singular jets (Proposition~\ref{prop:divisor}); this makes precise in what sense the main theorems and Deng's results \cite{Deng20} do not meet. Section~\ref{sec:quotient} constructs, on a threefold, a nonzero quotient bundle $Q_q$ of $\E{3}{5q}$ embedded in $\Sym^{2q}\Om\otimes\Sym^q\Om$, with surjections $\Sym^rQ_q\twoheadrightarrow Q_{qr}$. Section~\ref{sec:vanishing} recalls the Br\"uckmann--Rackwitz theorem. Section~\ref{sec:nonample} proves Theorem~\ref{thm:main}.

Section~\ref{sec:general} extends the construction to all dimensions, using Popov's polarisation theorem as in Rousseau's proof, and proves Theorem~\ref{thm:ci}. Section~\ref{sec:k4} treats order four, using Merker's computation of the invariant $4$-jet algebra in dimension $4$, and proves Theorem~\ref{thm:k4}. Section~\ref{sec:remarks} collects final remarks on the scope of the argument: the role of the residue class of $m$ modulo $2k-1$, with the one further case $E_{3,8}T^*_Z$ which the method reaches (Proposition~\ref{prop:eight}), the dependence of ampleness on the weight $m$, order two, dimension two, and orders $k\ge5$.

\subsection{What is new and what is not}\label{sec:new}
Let us say clearly, from the outset, where the novelty lies. All the ingredients are in the literature: Rousseau's generators of the invariant $3$-jet algebra, his description of $\Gr^\bullet\E{3}{m}$ and his use of Popov's theorem \cite{Rou06}; Merker's generators in order $4$ \cite{Mer10}; the Br\"uckmann--Rackwitz vanishing theorem \cite{BR90}; the strategy of \cite{Div08} of applying the latter to the graded pieces of a filtered jet bundle; and Etesse's remark that Schur powers with few rows are never ample on complete intersections of small codimension \cite{Ete21}.

Each of these had been available for years, and the row-count of the graded pieces of $E_{k,m}T^*_Z$ is something every reader of \cite{Div08} has in front of him.

The one idea which is not in the literature, and which is due to ChatGPT-6 Astra (see the Genesis), is to look at the filtration of $E_{k,m}T^*_Z$ by the degree in the highest derivative \emph{from the top rather than from the bottom}: the piece of maximal degree in $f^{(k)}$ is a \emph{quotient} of $E_{k,m}T^*_Z$, ampleness passes to quotients, and for $k=3$ and $5\mid m$ this quotient has only two rows. Around that idea, the extension to all dimensions (Popov), the order-four case (Merker's list), the discussion of the weight $m$ and the vanishing lemmas are an assembly of known facts.

Our own contribution, besides the verification and the generalisations, is a change of viewpoint: the recognition that this quotient is the exact analogue, for jet differentials, of the \emph{principal symbol} of a differential operator with respect to its order (Section~\ref{sec:symbol}). The graded pieces of the filtration by the highest derivative are of course classical \cite[(6.4)]{Dem97}; what is new is to use the top piece of the invariant bundle as a quotient, to read it as a symbol in the sense of differential operators, and to make this systematic. As far as Theorem~\ref{thm:main} is concerned it adds no new fact to the original argument, but it is the viewpoint from which the whole note is written.

It explains why the top piece is a quotient and not merely a subquotient; why it lies in a tensor bundle for $k=3$ (the symbols involve $f'$ only), whereas for $k=4$ it lies in a bundle of $2$-jet differentials (the symbols involve $f''$); why in both cases the number of rows of its irreducible summands is bounded by one plus the number of derivatives occurring in the symbols, so that the Br\"uckmann--Rackwitz theorem applies; and it gives the general criterion of Theorem~\ref{thm:criterion}, of which Theorems~\ref{thm:main}, \ref{thm:ci} and~\ref{thm:k4} are the cases in which the hypothesis has been verified. We would be grateful for any reference in which the quotient observation, or a stronger one, already appears.

\section{The principal symbol of a jet differential}\label{sec:symbol}

The construction on which all the results of this note rest is the analogue, for jet differentials, of the principal symbol of a linear differential operator. The analogy is not needed in the proofs, which are self-contained, but it is what makes the construction clear: it explains why the piece of maximal degree in the highest derivative is a \emph{quotient}, why it is sometimes a tensor bundle and sometimes not, and how the arguments of \S\ref{sec:quotient} and of \S\ref{sec:k4} are two instances of a single mechanism. It also yields a general non-ampleness criterion (Theorem~\ref{thm:criterion}), of which the main theorems are the cases in which the hypothesis can be verified, and it has a geometric counterpart on the Demailly--Semple tower, where the top symbol becomes the restriction to a thickening of the vertical divisor (\S\ref{sec:tower}).

This point of view is not in the original argument, which uses the word ``symbol'' for the top homogeneous component and checks by hand that it is a tensor in the case of $E_{3,5q}$ (see \S\ref{sec:new}); it is the way we came to understand that argument, and the framework in which its generalisations fit.

\subsection{The classical model}\label{sec:classical}
Let $X$ be a complex manifold and let $\Diff^{\le r}_X$ denote the sheaf of linear differential operators of order $\le r$ acting on holomorphic functions. In local coordinates such an operator is $P=\sum_{|\alpha|\le r}a_\alpha(x)\,\partial_x^\alpha$. Under a change of coordinates $y=\phi(x)$, with Jacobian $J=d\phi$ and second derivative $H=d^2\phi$, one has
\begin{equation}\label{eq:diffchain}
\partial_{x_i}=\sum_jJ_{ji}\,\partial_{y_j},\qquad
\partial_{x_i}\partial_{x_l}=\sum_{j,p}J_{ji}J_{pl}\,\partial_{y_j}\partial_{y_p}+\sum_jH_{j,il}\,\partial_{y_j},
\end{equation}
and in general $\partial_x^\alpha$ is the image of $\partial_y^{\,\bullet}$ under $\Sym^{|\alpha|}J$ plus operators of order $<|\alpha|$ whose coefficients involve the derivatives of order $\ge2$ of $\phi$. Three things follow at once. First, the filtration by the order is intrinsic: $\Diff^{\le r-1}_X$ is a subsheaf of $\Diff^{\le r}_X$. Second, the part of maximal order of $P$ transforms through $J$ alone, \textsl{i.e.}\ as a tensor: it is a section of $\Sym^rT_X$, the \emph{principal symbol}
\[
\sigma_r(P)(\xi)=\sum_{|\alpha|=r}a_\alpha\,\xi^\alpha ,
\]
a homogeneous polynomial function of degree $r$ on $T^*_X$. Third, the symbol is obtained by discarding the lower-order terms, so it is a \emph{quotient}: the symbol sequence
\begin{equation}\label{eq:symbolseq}
0\longrightarrow\Diff^{\le r-1}_X\longrightarrow\Diff^{\le r}_X\xrightarrow{\ \sigma_r\ }\Sym^rT_X\longrightarrow0
\end{equation}
is exact. Moreover the symbol is multiplicative, $\sigma_{r+s}(PQ)=\sigma_r(P)\sigma_s(Q)$. None of this is true of the lower-order coefficients: for $|\alpha|<r$ the coefficient $a_\alpha$ is not a tensor, there is no intrinsic subsheaf of ``operators of pure order $r$'' complementing $\Diff^{\le r-1}_X$, and $\Diff^{\le r}_X$ is not the direct sum of its graded pieces $\Sym^iT_X$, $i\le r$, unless one chooses a connection to identify them.

\subsection{Jet differentials and the filtration by the highest derivative}
We now describe the exact analogue for jet differentials. Let $X$ be a smooth variety of dimension $n$ and let $E^{GG}_{k,m}T^*_X$ be the Green--Griffiths bundle of (not necessarily invariant) jet differentials of order $k$ and weighted degree $m$: polynomials $P(f',\dots,f^{(k)})$ of weight $m$ with respect to $\deg f^{(j)}=j$, introduced by Green and Griffiths \cite{GG80} (see also \cite[\S6]{Dem97}, \cite[\S2.1]{Div08}). Its transition functions are induced by the chain rule for a change of coordinates $y=\phi(x)$, which by the formula of Fa\`a di Bruno reads
\begin{equation}\label{eq:faadibruno}
\widetilde{f^{(k)}}=J\,f^{(k)}+R_k\big(f',\dots,f^{(k-1)}\big),
\end{equation}
where $R_k$ is a polynomial of weight $k$ in the lower jet variables whose coefficients are the derivatives of order $2,\dots,k$ of $\phi$; for $k=1,2,3$ these are the formulas \eqref{eq:chainrule} of \S\ref{sec:quotient}. The invariant bundle $E_{k,m}T^*_X\subset E^{GG}_{k,m}T^*_X$ is the subbundle of polynomials satisfying $P(j_k(f\circ\varphi))=\varphi'(0)^mP(j_kf)$ for every germ of biholomorphism $\varphi$ of $(\CC,0)$; the reparametrisation action is also of the form \eqref{eq:faadibruno}, with $\varphi'^k$ in place of $J$.

Formula \eqref{eq:faadibruno} has the same shape as \eqref{eq:diffchain}: the highest derivative appears linearly, through $J$, and the correction terms involve only lower derivatives of $f$ and higher derivatives of $\phi$. The role of the order of an operator is therefore played by the \emph{degree in the highest derivative} $f^{(k)}$, and the role of the derivatives of $J$ by the terms $R_k$. For $0\le j\le m/k$ let
\[
F_jE^{GG}_{k,m}T^*_X:=\{P:\ \deg_{f^{(k)}}P\le j\},\qquad F_jE_{k,m}T^*_X:=F_jE^{GG}_{k,m}T^*_X\cap E_{k,m}T^*_X,
\]
and let $\sigma_j(P)$ be the homogeneous component of degree $j$ in $f^{(k)}$ of $P\in F_j$. We regard $\sigma_j(P)$ as a polynomial in $(f',\dots,f^{(k-1)})$ with values in the homogeneous polynomials of degree $j$ in a vector variable $\xi\in T_X$, the variable $\xi$ standing for $f^{(k)}$:
\[
\sigma_j(P)\big(f',\dots,f^{(k-1)};\xi\big)=[z^j]\,P\big(f',\dots,f^{(k-1)},z\xi\big),
\]
where $z$ is an auxiliary scalar variable and, following the usual convention, $[z^j]F$ denotes the coefficient of $z^j$ in a polynomial $F$ in $z$. Finally, we denote by $q(k,m)$ the maximal degree in $f^{(k)}$ of an invariant jet differential of order $k$ and weight $m$, that is, the largest $j$ with $F_jE_{k,m}T^*_X\neq F_{j-1}E_{k,m}T^*_X$; it is well defined, since the fibre of $E_{k,m}T^*_X$ is the same subspace of polynomials in every trivialisation, and $q(k,m)\le m/k$ for weight reasons.

\begin{proposition}[The symbol of a jet differential]\label{prop:symbol}
Let $k\ge2$, $m\ge1$ and $0\le j\le m/k$.
\begin{enumerate}
\item $F_jE^{GG}_{k,m}T^*_X$ and $F_jE_{k,m}T^*_X$ are subbundles, and $F_0\subset F_1\subset\cdots$ are increasing filtrations by subbundles.
\item The symbol $\sigma_j$ induces an isomorphism of vector bundles \textup{(}\cite[(6.4)]{Dem97}\textup{)}
\[
\Gr^F_jE^{GG}_{k,m}T^*_X\ \xrightarrow{\ \sim\ }\ E^{GG}_{k-1,m-kj}T^*_X\otimes\Sym^j\Om
\]
and an injective morphism of vector bundles, with image a subbundle,
\begin{equation}\label{eq:symbolinv}
\Gr^F_jE_{k,m}T^*_X\ \lhook\joinrel\longrightarrow\ E_{k-1,m-kj}T^*_X\otimes\Sym^j\Om .
\end{equation}
In words: the symbol of an invariant jet differential of order $k$ and weight $m$ is an invariant jet differential of order $k-1$ and weight $m-kj$ with values in $\Sym^j\Om$. \textup{(}For $k-1=1$ one has $E_{1,p}=\Sym^p\Om$, and for $p=0$ both $E^{GG}_{k-1,0}$ and $E_{k-1,0}$ are $\cO_X$.\textup{)}
\item The symbol is multiplicative: multiplication of polynomials defines a morphism of vector bundles
\[
F_jE_{k,m}T^*_X\otimes F_{j'}E_{k,m'}T^*_X\longrightarrow F_{j+j'}E_{k,m+m'}T^*_X ,
\]
and $\sigma_{j+j'}(PP')=\sigma_j(P)\,\sigma_{j'}(P')$, the product on the right being that of polynomials in $(f',\dots,f^{(k-1)};\xi)$. In particular, the product of nonzero symbols is a nonzero symbol.
\item Let $q=q(k,m)$. Then $F_q=E_{k,m}T^*_X$, and the \emph{top symbol} $\sigma_q$ identifies the quotient bundle $E_{k,m}T^*_X/F_{q-1}$ with a subbundle of
\[
E_{k-1,m-kq}T^*_X\otimes\Sym^q\Om .
\]
\end{enumerate}
\end{proposition}

\begin{proof}
(1) Since \eqref{eq:faadibruno} is affine-linear in $f^{(k)}$, the substitution does not increase the degree in $f^{(k)}$, so the fibres of $F_j$ are the same subspace of polynomials in every trivialisation and are preserved by the transition functions; the same holds for $F_jE_{k,m}$, since the fibres of $E_{k,m}$ are the same subspace (the invariants) in every trivialisation.

(2) Let $P\in F_j$. Substituting \eqref{eq:faadibruno} and the chain rule for the lower jets into $P$ and keeping the component of degree $j$ in $f^{(k)}$, only $Jf^{(k)}$ contributes from \eqref{eq:faadibruno}, because $R_k$ has degree $0$ in $f^{(k)}$, and only $\sigma_j(P)$ contributes from $P$, because the components of lower degree remain of lower degree. Hence
\[
\sigma_j\big(P\circ(\text{chain rule of order }k)\big)=\sigma_j(P)\circ\big(\text{chain rule of order }k-1,\ \xi\mapsto J\xi\big),
\]
the chain rule of order $k-1$ acting on $(f',\dots,f^{(k-1)})$. This is precisely the transition rule of $E^{GG}_{k-1,m-kj}\otimes\Sym^j\Om$ (polynomials of weight $m-kj$ in the $(k-1)$-jet, homogeneous of degree $j$ in the tangent vector $\xi$). The map $F_j\to E^{GG}_{k-1,m-kj}\otimes\Sym^j\Om$ so defined is therefore a morphism of bundles; its kernel is $F_{j-1}$, and it is onto, since every polynomial of the right bidegree occurs as the degree-$j$ part of some $P$. This proves the first assertion.

For the second, let $P\in F_jE_{k,m}$, write $P=\sum_{i\le j}P_i$ with $P_i$ homogeneous of degree $i$ in $f^{(k)}$, so that $\sigma_j(P)=P_j$, and let $\varphi$ be a reparametrisation. Comparing the components of degree $j$ in $f^{(k)}$ on both sides of $P(j_k(f\circ\varphi))=\varphi'^mP(j_kf)$, and using that the reparametrised $k$-th derivative is $\varphi'^kf^{(k)}$ plus terms of lower order, one finds
\[
\varphi'^{kj}\,P_j\big(j_{k-1}(f\circ\varphi),\xi\big)=\varphi'^m\,P_j\big(j_{k-1}f,\xi\big)\qquad\text{for every }\xi ,
\]
that is, for fixed $\xi$ the polynomial $P_j(\,\cdot\,,\xi)$ is an invariant jet differential of order $k-1$ and weight $m-kj$. Thus the image of $F_jE_{k,m}$ lies in $E_{k-1,m-kj}\otimes\Sym^j\Om$; it is a subbundle because, in every trivialisation, it is the fixed subspace $\sigma_j(F_jE_{k,m,x})$, so its rank is constant.

(3) If $P$ and $P'$ are invariant of weights $m$ and $m'$, then $PP'$ is invariant of weight $m+m'$, directly from the definition. The transition functions of all the bundles $E^{GG}_{k,\bullet}T^*_X$ are given by one and the same substitution \eqref{eq:faadibruno} of the jet variables, which is a homomorphism of polynomial rings; hence multiplication commutes with them and is a morphism of bundles.

Write $P=\sum_{i\le j}P_i$ and $P'=\sum_{i'\le j'}P'_{i'}$ with $P_i,P'_{i'}$ homogeneous of degrees $i,i'$ in $f^{(k)}$. Then $PP'=\sum P_iP'_{i'}$, and the only term of degree $j+j'$ in $f^{(k)}$ is $P_jP'_{j'}$, all the others having $i+i'<j+j'$; thus $PP'\in F_{j+j'}$ and
\[
\sigma_{j+j'}(PP')=P_jP'_{j'}=\sigma_j(P)\,\sigma_{j'}(P').
\]
The last assertion holds because a polynomial ring is an integral domain.

(4) $F_q=E_{k,m}$ by definition of $q$, and the rest is (2) with $j=q$.
\end{proof}

\begin{example}[The Wronskian]\label{ex:wronskian}
Let $k=2$ and $m=3$, and write $v=f'$, $a=f''$. The fibre of $E_{2,3}T^*_X$ is spanned by the cubic monomials in $v$ and by the Wronskians $W_{ij}=v_ia_j-v_ja_i$ \cite[\S12]{Dem97}. A change of coordinates gives
\[
\widetilde W_{ij}=\big(Jv\wedge(Ja+H(v,v))\big)_{ij}=(Jv\wedge Ja)_{ij}+\big(Jv\wedge H(v,v)\big)_{ij},
\]
and the last term is cubic in $v$, \textsl{i.e.}\ lies in $F_0=\Sym^3\Om$. Thus the span of the Wronskians is not preserved by the transition functions, but the symbols of the Wronskians transform as a tensor: the symbol of $W_{ij}$ is $\xi\mapsto v_i\xi_j-v_j\xi_i$, the $2$-form $v\wedge\xi$, and the symbol sequence of Proposition~\ref{prop:symbol} reads
\[
0\longrightarrow\Sym^3\Om\longrightarrow E_{2,3}T^*_X\xrightarrow{\ \sigma_1\ }\Lambda^2\Om\longrightarrow0,
\]
the exact analogue of \eqref{eq:symbolseq}.
\end{example}

\begin{example}[Order three]\label{ex:orderthree}
Let $k=3$ and write $b=f'''$. The generator $A_{ij;\ell}=v_\ell(v_ib_j-v_jb_i)-3a_\ell W_{ij}$ of Rousseau's theorem (Theorem~\ref{thm:rousseau} below) has degree $1$ in $b$, and its symbol $v_\ell(v_i\xi_j-v_j\xi_i)$ does not involve $a$: the correction term $-3a_\ell W_{ij}$, which is needed for invariance under reparametrisation, has degree $0$ in $b$ and is invisible in the symbol. By contrast, the symbol of the third-order Wronskian $D=\det(v,a,b)$ is $\det(v,a,\xi)$, which does involve $a$. This dichotomy is what the rest of the note exploits: at the weights $5q$ only the generators $A_{ij;\ell}$ reach the top of the filtration (\S\ref{sec:quotient}), and their symbols are tensors.
\end{example}

\begin{remark}[Relation with Demailly's filtration]\label{rem:demfiltration}
The filtration $F_\bullet$ is the coarsest step of the multi-filtration introduced by Demailly \cite[(6.3)--(6.5) and Def.~6.7]{Dem97} (see also \cite[\S2.1]{Div08}): in Demailly's notation, $F_jE_{k,m}=F^{\,m-kj}_{k-1}E_{k,m}$ is the subbundle of polynomials whose monomials have partial weighted degree of order $k-1$ at least $m-kj$. Iterating Proposition~\ref{prop:symbol}(2), \textsl{i.e.}\ filtering the target $E_{k-1,m-kj}$ by the degree in $f^{(k-1)}$ and so on, one recovers the full multi-filtration and its graded bundle $\Gr^\bullet E_{k,m}T^*_X$, whose pieces are subbundles of the tensor bundles
\[
\Sym^{\ell_1}\Om\otimes\dots\otimes\Sym^{\ell_k}\Om,\qquad \ell_1+2\ell_2+\dots+k\ell_k=m .
\]
This is what is used in \cite{Dem97,Div08,Rou06} to compute Euler characteristics and to prove vanishing theorems, for which only the graded object matters.

What is used in this note is different: only the coarsest filtration, and the fact, forgotten when one passes to the graded object, that its top piece is a \emph{quotient} of $E_{k,m}T^*_X$ and not merely a subquotient.
\end{remark}

\begin{remark}[The filtration does not split]\label{rem:nosplit}
As for differential operators, the filtration $F_\bullet$ has no canonical splitting, and in general no splitting at all. The simplest case is instructive. For $k=2$ and $m=2$, in local coordinates the fibre of $E^{GG}_{2,2}T^*_X$ consists of the polynomials
\[
P(f',f'')=\sum_{i,j}\beta_{ij}\,f'_if'_j+\sum_i\lambda_i\,f''_i,\qquad \beta_{ij}=\beta_{ji},\ \lambda_i\in\CC
\]
(and so does the fibre of $E^{GG}_{k,2}T^*_X$ for every $k\ge2$, since $f^{(j)}$ has weight $j$); here $F_0=\Sym^2\Om$, and the symbol sequence is $0\to\Sym^2\Om\to E^{GG}_{2,2}T^*_X\to\Om\to0$. The map which sends the $2$-jet at $x$ of a holomorphic function $g$ to the jet differential
\[
j_2f\longmapsto(g\circ f)''(0)=\sum_{i,j}\frac{\partial^2g}{\partial x_i\partial x_j}(x)\,f'_if'_j+\sum_i\frac{\partial g}{\partial x_i}(x)\,f''_i
\]
identifies $E^{GG}_{2,2}T^*_X$ with the bundle $J^2(\cO_X)/\cO_X$ of $2$-jets of functions modulo constants, and the symbol sequence with the classical jet sequence $0\to\Sym^2\Om\to J^2(\cO_X)/\cO_X\to\Om\to0$, whose splittings correspond to holomorphic torsion-free affine connections on $X$; these exist on complex tori, but on a compact K\"ahler manifold $X$ the existence of a holomorphic connection on $T_X$ forces all the Chern classes $c_i(X)$, $i\ge1$, to vanish in $H^{2i}(X,\mathbb{R})$ \cite{Ati57}. This excludes every smooth hypersurface $X\subset\PP^{n+1}$ of degree $d$ with $n\ge2$: if $h$ is the hyperplane class, then $c_1(X)=(n+2-d)\,h$, and when $d=n+2$ one has $c_2(X)=\binom{n+2}{2}h^2\neq0$.

Thus the top piece is a quotient but not a direct summand, exactly as $\Sym^rT_X$ is a quotient but not a direct summand of $\Diff^{\le r}_X$. (On the invariant bundle the question does not arise in this case, since $E_{k,2}T^*_X=\Sym^2\Om$: the non-tensorial terms $\sum_i\lambda_if''_i$ are precisely what invariance under reparametrisation removes, \textsl{cf.}\ \S\ref{sec:conjecture}.)
\end{remark}

\subsection{When the top symbol is a tensor}
There is one point where the analogy with differential operators breaks down, and it is the decisive one. The symbol of a linear differential operator has values in a tensor bundle. The symbol of a jet differential of order $k\ge3$ has values in $E_{k-1,m-kj}T^*_X\otimes\Sym^j\Om$, and $E_{k-1,m-kj}T^*_X$ is itself only a filtered bundle: the ``coefficients'' of the top component are polynomials in the intermediate derivatives $f'',\dots,f^{(k-1)}$, which transform non-linearly. The symbol takes values in a tensor bundle when it does not involve them.

We shall use repeatedly the following elementary fact, which we state once and for all. By a \emph{tensor bundle} we mean a bundle $T$ associated with the frame bundle of $X$ and a polynomial representation $T_0$ of $\GL_n(\CC)$, such as $\Sym^{\ell_1}\Om\otimes\dots\otimes\Sym^{\ell_s}\Om$; its transition functions are given by the action of the Jacobian $J$ on $T_0$. Recall that $T_0$ decomposes canonically into isotypic components, $T_0=\bigoplus_\lambda\Sch{\lambda}V^*\otimes M_\lambda$, where $M_\lambda$ is the multiplicity space of $\Sch{\lambda}V^*$.

\begin{lemma}[Equivariant subbundles]\label{lem:equivariant}
Let $T$ be a tensor bundle with fibre $T_0$, and let $S\subset T$ be a subbundle whose fibre is, in every trivialisation, one and the same $\GL_n(\CC)$-stable subspace $S_0\subset T_0$. Then $S_0=\bigoplus_\lambda\Sch{\lambda}V^*\otimes M'_\lambda$ with $M'_\lambda\subset M_\lambda$, and accordingly
\[
S=\bigoplus_\lambda\Sch{\lambda}\Om\otimes M'_\lambda
\]
is a direct sum of Schur bundles occurring in $T$. In particular every isotypic component of $S$ is a direct summand of $S$, and every irreducible summand of $S$ is a $\Sch{\lambda}\Om$ with $\Sch{\lambda}V^*$ occurring in $T_0$.
\end{lemma}

\begin{proof}
Since the isotypic decomposition of $T_0$ is preserved by every $\GL_n(\CC)$-equivariant automorphism, it defines a decomposition of $T$ into the subbundles $\Sch{\lambda}\Om\otimes M_\lambda$, where $M_\lambda$ denotes the trivial bundle with that fibre. A $\GL_n(\CC)$-stable subspace $S_0$ is the direct sum of its intersections with the isotypic components of $T_0$, and $S_0\cap(\Sch{\lambda}V^*\otimes M_\lambda)=\Sch{\lambda}V^*\otimes M'_\lambda$ for a subspace $M'_\lambda\subset M_\lambda$, by Schur's lemma. Since $S_0$ is the same in every trivialisation, so are the $M'_\lambda$, and the decomposition of $S_0$ is a decomposition of $S$ into subbundles.
\end{proof}

\begin{corollary}[Tensorial symbols]\label{cor:tensorial}
Let $k\ge2$, $m\ge1$, let $q=q(k,m)\ge1$ be the maximal degree in $f^{(k)}$ of an invariant jet differential of order $k$ and weight $m$, and let $1\le s\le k-1$. Suppose that the top symbols $\sigma_q(P)$, $P\in E_{k,m}T^*_X$, involve only the derivatives $f',\dots,f^{(s)}$. Then the top symbol identifies the quotient bundle $E_{k,m}T^*_X/F_{q-1}$ with a subbundle of
\[
E_{s,m-kq}T^*_X\otimes\Sym^q\Om ,
\]
and every irreducible summand of a graded piece of this quotient \textup{(}for the filtration induced by the degrees in $f'',\dots,f^{(s)}$\textup{)} is a Schur bundle $\Sch{\lambda}\Om$ with $\lambda$ nonempty and having at most $s+1$ rows. In particular, if $s=1$, then $E_{k,m}T^*_X/F_{q-1}$ is a subbundle of the tensor bundle $\Sym^{m-kq}\Om\otimes\Sym^q\Om$.
\end{corollary}

\begin{proof}
Polynomials which do not involve $f^{(s+1)},\dots,f^{(k-1)}$ form the subbundle $E^{GG}_{s,p}\subset E^{GG}_{k-1,p}$, since the chain rule for $(f',\dots,f^{(s)})$ introduces no higher derivative, and $E_{s,p}=E_{k-1,p}\cap E^{GG}_{s,p}$ by definition of invariance. The first assertion is therefore Proposition~\ref{prop:symbol}(4).

For the second, $E_{s,p}\otimes\Sym^q\Om$ is filtered by the degrees in $f'',\dots,f^{(s)}$ (Remark~\ref{rem:demfiltration}), with graded pieces subbundles of
\[
\Sym^{\ell_1}\Om\otimes\dots\otimes\Sym^{\ell_s}\Om\otimes\Sym^q\Om ;
\]
by Pieri's rule such a tensor product of $s+1$ symmetric powers only contains $\Sch{\lambda}$ with at most $s+1$ rows, and $\lambda$ is nonempty since $q\ge1$. The image of the top symbol is, in every trivialisation, one and the same $\GL_n$-stable subspace, since the fibre of $E_{k,m}$ is the same space of invariants in every trivialisation, the transition functions restrict to the $\GL_n$-action on it when $\phi$ is linear, and the symbol is equivariant; the same holds for its graded pieces, and Lemma~\ref{lem:equivariant} gives the assertion on the irreducible summands. For $s=1$, $E_{1,p}=\Sym^p\Om$ and there is no further filtration.
\end{proof}

In short: \emph{if the top symbols involve only the first $s$ derivatives $f',\dots,f^{(s)}$, the irreducible summands of the top quotient have at most $s+1$ rows}. The Br\"uckmann--Rackwitz theorem (Theorem~\ref{thm:BR} below) kills the sections of Schur bundles with few rows on complete intersections of small codimension, and ampleness descends to quotients. Putting the three facts together gives the criterion which underlies all the theorems of this note.

\begin{theorem}[Non-ampleness criterion]\label{thm:criterion}
Let $Z\subset\PP^N$ be a smooth complete intersection of $c$ hypersurfaces of degrees $\ge2$. Let $k\ge2$, $m\ge1$, let $q=q(k,m)\ge1$ be the maximal degree in $f^{(k)}$ of an invariant jet differential of order $k$ and weight $m$, and suppose that the top symbols of the elements of $E_{k,m}T^*_Z$ involve only $f',\dots,f^{(s)}$ for some $1\le s\le k-1$ with
\[
(s+2)\,c<N .
\]
Then $E_{k,m}T^*_Z$ is not ample.
\end{theorem}

\begin{proof}
Set $E=E_{k,m}T^*_Z$ and $Q=E/F_{q-1}E$, a nonzero quotient bundle (nonzero since the degree $q$ is attained). Suppose $E$ ample. Then $Q$ is ample by (H1), and $\Sym^\ell Q$ is generated by global sections for some $\ell\ge1$ by (H2). By Proposition~\ref{prop:symbol}(3), multiplication followed by the symbol of degree $\ell q$ defines a morphism of bundles
\[
\mu_\ell\colon\ \Sym^\ell Q\longrightarrow\Gr^F_{\ell q}E_{k,\ell m}T^*_Z\ \lhook\joinrel\longrightarrow\ E_{k-1,\ell(m-kq)}T^*_Z\otimes\Sym^{\ell q}\OmZ ,
\]
well defined on $\Sym^\ell Q$ because a product with a factor in $F_{q-1}$ has degree $<\ell q$ in $f^{(k)}$, and nonzero because the product of $\ell$ copies of a nonzero top symbol is nonzero. Its image consists of products of top symbols, hence of polynomials in $f',\dots,f^{(s)}$ only, so $\mu_\ell$ takes values in
\[
T_\ell:=E_{s,\ell(m-kq)}T^*_Z\otimes\Sym^{\ell q}\OmZ .
\]
By (H3) the image of $\mu_\ell$ is a nonzero globally generated subsheaf of $T_\ell$, so $H^0(Z,T_\ell)\neq0$.

On the other hand $T_\ell$ is filtered by subbundles whose graded pieces are direct sums of Schur bundles $\Sch{\lambda}\OmZ$ with $\lambda$ nonempty and at most $s+1$ rows (as in the proof of Corollary~\ref{cor:tensorial}, with $\ell q\ge1$ in place of $q$). For such $\lambda$ one has $\sum_{j\le c}\lambda'_j\le(s+1)c<N-c=\dim Z$, so $H^0(Z,\Sch{\lambda}\OmZ)=0$ by Theorem~\ref{thm:BR}, and $H^0(Z,T_\ell)=0$ by the filtration lemma \cite[Lemma~2]{Div08}. This is a contradiction.
\end{proof}

Note that $\ell q$ need not be the maximal degree in $f^{(k)}$ at weight $\ell m$, so that $\Gr^F_{\ell q}E_{k,\ell m}$ need not be a quotient of $E_{k,\ell m}T^*_Z$; this is harmless, since the argument only needs a nonzero map from $\Sym^\ell Q$ to a bundle without sections. When $\ell q$ is the maximal degree, as happens for $k=3$ and $m=5q$, the map $\mu_\ell$ lands in the top quotient of $E_{k,\ell m}T^*_Z$ and is onto it; see Proposition~\ref{prop:quotient}.

\begin{remark}[Below the threshold]\label{rem:below}
The hypothesis of Theorem~\ref{thm:criterion} is empty for $s=k-1$, and the criterion then says that $E_{k,m}T^*_Z$ is not ample whenever $q(k,m)\ge1$ and $(k+1)c<N$, that is, below the threshold $k\ge N/c-1$ of Conjecture~\ref{conj:DT}. Since $E_{k,m}=E_{k-1,m}$ when $q(k,m)=0$, induction on $k$, starting from the case $k=1$ in which $\Sym^m\OmZ$ is not ample when $2c<N$ (Schneider's theorem, the case $\lambda=(m)$ of Theorem~\ref{thm:BR}), gives: \emph{if $k<N/c-1$, then $E_{k,m}T^*_Z$ is not ample for any $m\ge1$}.

This is not new, and does not need the symbol: below the threshold $Z$ carries no invariant jet differentials of order $k$ at all \cite[Thm.~7]{Div08}, and if $E_{k,m}T^*_Z$ were ample, the image of the nonzero multiplication map $\Sym^\ell E_{k,m}T^*_Z\to E_{k,\ell m}T^*_Z$ would be a nonzero globally generated subsheaf for $\ell\gg0$, which is Demailly's argument \cite[7.7~iii)]{Dem97}. The threshold of the conjecture is therefore necessary.

What happens at the threshold is different: there $E_{k,m}T^*_Z$ does have sections for large degrees, in fact many of them, and the point of the theorems of this note is that their top symbols all vanish, since by the proof of Theorem~\ref{thm:criterion} the target of the top symbol has no nonzero sections.
\end{remark}

\subsection{Reading the main theorems in this light}\label{sec:reading}
The hypothesis of Theorem~\ref{thm:criterion} is a statement about the invariant algebra at a point, and it can be checked whenever a system of generators of that algebra is known: the top symbol of an invariant of weight $m$ is a linear combination of products of top symbols of generators, one product for each monomial in the generators of weight $m$ and degree $q$ in $f^{(k)}$, and a count of weights determines which monomials these are.

When no generator has ratio $\deg_{f^{(k)}}/\operatorname{weight}$ larger than $1/(2k-1)$, as is the case for $k\le4$, at the weights which are multiples of $2k-1$ these monomials are the products of generators achieving the extremal ratio $1/(2k-1)$, and it suffices to look at the top symbols of those. This is what the rest of the note does.
\begin{itemize}
\item \emph{Order $2$}, every $m$. Here $s=k-1=1$ and the hypothesis is empty: the generators are $f'_i$ and the Wronskians $W_{ij}$, and the top symbol of $W_{ij}$ in $f''$ is $\xi\mapsto f'_i\xi_j-f'_j\xi_i$. Hence $E_{2,m}T^*_Z$ is not ample for $3c<N$ and every $m$ (Remark~\ref{rem:ordertwo}), which is Remark~\ref{rem:below} for $k=2$. In the simplest case, $0\to\Sym^3\Om\to E_{2,3}T^*_X\to\Lambda^2\Om\to0$ is the symbol sequence, the exact analogue of \eqref{eq:symbolseq} (Example~\ref{ex:wronskian}).
\item \emph{Order $3$}, $m=5q$. By Rousseau's theorem and the weight count of \S\ref{sec:quotient}, the top symbols are the products of $q$ symbols $v_\ell(v_ib_j-v_jb_i)$ of the generators $A_{ij;\ell}$, which involve $f'$ only: $s=1$, and Theorems~\ref{thm:main} and~\ref{thm:ci} follow, with the top quotient $\Sch{(2q,q,0,\dots,0)}\OmZ\subset\Sym^{2q}\OmZ\otimes\Sym^q\OmZ$ (Proposition~\ref{prop:quotient}, Remark~\ref{rem:schur}). The correction term $-3a_\ell W_{ij}$ of $A_{ij;\ell}$, which involves $f''$, is exactly the kind of term that the symbol discards (Example~\ref{ex:orderthree}).
\item \emph{Order $4$}, $m=7q+r$ with $r\in\{0,1,2\}$. By Merker's theorem and the defect count of \S\ref{sec:k4}, the top symbols involve $f'$ and $f''$ but not $f'''$: $s=2$, and Theorem~\ref{thm:k4} follows. Here the top quotient is naturally a subbundle of $E_{2,3q+r}T^*_Z\otimes\Sym^q\OmZ$ rather than of a tensor bundle, with irreducible summands of at most three rows, and this is why order $4$ requires $4c<N$ instead of $3c<N$.
\item \emph{Why not every weight.} At weight $m$ with $5\nmid m$ (order $3$), or $m\equiv3,\dots,6\pmod7$ (order $4$), the weight left over by the extremal generators can be spent on generators of smaller ratio whose top symbols involve all the intermediate derivatives (for instance, in order $3$, the third-order Wronskian $D$, as in $D\cdot A_{12;1}^{\,q-1}$ at weight $5q+1$, and the Wronskian $W^{10}$ in order $4$): then only $s=k-1$ is possible, and the criterion requires $(k+1)c<N$, which is the range below the threshold of the conjecture, where non-ampleness was already known (Remark~\ref{rem:below}); in the range of the conjecture it gives nothing, in accordance with Remark~\ref{rem:5divides} (with the sporadic exception of the weight $8$ in order $3$, where iterating the symbol still isolates a two-row quotient, Proposition~\ref{prop:eight}).
\end{itemize}
The Wronskian $\det(f',\dots,f^{(k)})$, the source of the summands with $k$ rows, has the ratio $2/(k(k+1))$, smaller than $1/(2k-1)$ for $k\ge3$ (Remark~\ref{rem:whyrows}): this is why it stays out of the top symbol at the weights which are multiples of $2k-1$, and why the obstruction exists at all.
\subsection{The top symbol on the Demailly--Semple tower}\label{sec:tower}
The ``almost'' in Deng's theorem of almost $k$-jet ampleness has, in the language of the symbol, a precise meaning, which we now explain: the top symbol is the restriction of the jet differentials to an infinitesimal neighbourhood of the vertical divisor of the Demailly--Semple tower, which lies in the singular jets, whereas the positivity proved in \cite{Deng20} concerns the regular jets, where the symbol carries no information. This is the point of view from which the main theorems and Deng's results should be compared.

\begin{remark}[Deng's results and the role of singular jets]\label{rem:Deng}
The theorems of \S\ref{sec:statement} and Deng's results \cite{Deng20} look at the same objects from two sides which do not meet. Deng's Definition~1.2 calls $Z$ \emph{almost $k$-jet ample} if for some $(a_1,\dots,a_k)\in\NN^k$ the weighted tautological bundle $\cO_{Z_k}(a_k,\dots,a_1)$ on the Demailly--Semple tower is big with augmented base locus contained in the locus $Z_k^{\rm sing}$ of singular jets; his Theorem~C proves almost $k$-jet ampleness for $k\ge N/c-1$, and his Corollary~D produces an ample \emph{subbundle} $F\subset\Ez{k}{m}$ whose sections detect regular jets.

Both statements concern only what happens on regular jets, and this is intentional: the augmented base locus is allowed to fill $Z_k^{\rm sing}$, and a subbundle of $\Ez{k}{m}$ says nothing about the quotients of $\Ez{k}{m}$. Our theorems, on the contrary, concern only what happens on singular jets, and say nothing about regular ones.
\end{remark}

The obstruction found here lives exactly where the ``almost'' in \cite{Deng20} lives, and this can be made precise. Let $X$ be a smooth variety of dimension $n\ge2$, let $X_k=P_kT_X$ be its Demailly--Semple tower with projections $\pi_{k,j}$, and let $D_k=P(T_{X_{k-1}/X_{k-2}})\subset X_k$ be the vertical divisor, the zero divisor of the natural morphism $\cO_{X_k}(-1)\to\pi_{k,k-1}^*\cO_{X_{k-1}}(-1)$ \cite[(5.8)--(5.9)]{Dem97}, \cite[\S4.2]{Div08}; it is contained in the locus $X_k^{\rm sing}$ of singular jets \cite[(5.13)]{Dem97}. As before, $F_j\E{k}{m}$ denotes the filtration by the degree in $f^{(k)}$.

\begin{proposition}\label{prop:divisor}
For every $k\ge2$, $m\ge1$ and $0\le j\le m/k$,
\[
F_j\E{k}{m}=\pi_{k,0\,*}\big(\cO_{X_k}(m)\otimes\cO_{X_k}(-(m-j)\,D_k)\big) .
\]
Consequently, if $q$ is the maximal degree in $f^{(k)}$ of an invariant of weight $m$, then the extremal quotient $\E{k}{m}/F_{q-1}$ is the image of $\E{k}{m}$ in the direct image $\pi_{k,0\,*}\big(\cO_{X_k}(m)\otimes\cO_{(m-q+1)D_k}\big)$: it is the restriction of the jet differentials to the subscheme $(m-q+1)D_k$, the $(m-q)$-th infinitesimal neighbourhood of the vertical divisor $D_k\subset X_k^{\rm sing}$.
\end{proposition}

\begin{proof}
Apply the direct image formula \cite[Prop.~6.16~i)]{Dem97}, valid whenever the partial sums $b_i=a_1+\dots+a_i$ are nonnegative (this is what makes $\cO_{X_k}(a)$ a subsheaf of $\cO_{X_k}(m)$, \cite[(6.15)]{Dem97}), to $a=(0,\dots,0,m-j,j)$. One has $b=(0,\dots,0,m-j,m)$, so
\[
\cO_{X_k}(a)=\cO_{X_k}(m)\otimes\cO_{X_k}(-b\cdot D^\star)=\cO_{X_k}(m)\otimes\cO_{X_k}(-(m-j)D_k),
\]
and the conditions
\[
\ell_{s+1}+2\ell_{s+2}+\dots+(k-s)\ell_k\le a_{s+1}+\dots+a_k\qquad(0\le s\le k-1)
\]
reduce to $\ell_k\le j$ for $s=k-1$, the others being automatic since $\ell_1+2\ell_2+\dots+k\ell_k=m$.

The last assertion follows by applying $\pi_{k,0\,*}$ to the exact sequence
\[
0\to\cO_{X_k}(m)(-(m-q+1)D_k)\to\cO_{X_k}(m)\to\cO_{X_k}(m)\otimes\cO_{(m-q+1)D_k}\to0
\]
and using that $\pi_{k,0\,*}$ is left exact.
\end{proof}

\begin{remark}[The obstruction lives along the vertical divisor]\label{rem:vertical}
Thus the quotients on which our obstruction is read off are precisely the traces of the jet differentials along a thickening of the vertical divisor, which lies in $X_k^{\rm sing}$; on the regular jets, where Deng's positivity lives, they carry no information. This is consistent with \cite[Thm.~6.8~iii)]{Dem97}, according to which the relative base locus of $|\cO_{X_k}(m)|$ over $X$ is precisely $X_k^{\rm sing}$, and with the fact that $\cO_{X_k}(1)$ is relatively big over $X$ \cite[Thm.~6.8~iii)]{Dem97} but never relatively ample for $k\ge2$ \cite[\S6, after the proof of Thm.~6.8]{Dem97}. It is not even relatively nef: if $L$ is a line in a fibre of $X_{k-1}\to X_{k-2}$ and $C\subset D_k=P(T_{X_{k-1}/X_{k-2}})$ is the curve of its own tangent directions, then $\cO_{X_k}(-1)|_C=T_L$, so $\cO_{X_k}(1)\cdot C=-2$. This is why one works with the weighted bundles $\cO_{X_k}(a)$ of \cite[Prop.~6.16]{Dem97} in the first place.

In this sense the theorems should be read as a confirmation that ordinary ampleness of $\Ez{k}{m}$ is the wrong positivity notion for jet bundles of order $\ge2$, rather than as a failure of the geometric content of the conjecture.
\end{remark}

\section{The extremal quotient of the invariant third-jet bundle}\label{sec:quotient}

We now specialise the construction of \S\ref{sec:symbol} to order $k=3$, where everything can be made completely explicit thanks to Rousseau's theorem. The filtration of $\E{3}{m}$ by the degree in the third derivative $f'''$ is the filtration $F_\bullet$ of \S\ref{sec:symbol}; in this section we look at it from the top: for $m=5q$ we identify the piece of highest degree, which is a \emph{quotient} of $\E{3}{m}$, with an explicit two-row Schur power of $\Om$, and we show that these quotients are compatible with multiplication. This is the heart of the argument; everything else is a matter of feeding it into standard vanishing and ampleness theorems.

Throughout this section $X$ is a smooth complex threefold. In local coordinates we write
\[
v=f'(0),\qquad a=f''(0),\qquad b=f'''(0)\qquad(v,a,b\in\CC^3),
\]
and we write $v_i,a_i,b_i$ for their components.

\subsection{Rousseau's generators}
\begin{theorem}[{Rousseau, \cite[Th\'eor\`eme~1.1]{Rou06}}]\label{thm:rousseau}
In dimension three, the algebra $A_3=\bigoplus_mE_{3,m}$ of invariant $3$-jet differentials is generated by the following polynomials, where $1\le i<j\le3$ and $1\le\ell\le3$:
\begin{equation}\label{eq:gens}
\begin{gathered}
v_i,\qquad W_{ij}=v_ia_j-v_ja_i,\qquad D=\det(v,a,b),\\
A_{ij;\ell}=v_\ell\,(v_ib_j-v_jb_i)-3\,a_\ell\,W_{ij}.
\end{gathered}
\end{equation}
\end{theorem}

In \cite{Rou06} the generators $v_i$, $W_{ij}$, $A_{ij;\ell}$, $D$ are denoted $f_i'$, $w_{ij}$, $w_{ij}^\ell$, $W$; their weighted degrees are $1,3,5,6$, and their degrees in the variable $b$ are $0,0,1,1$. The theorem asserts that these elements generate the algebra; it does not assert that they are algebraically independent, and we shall never use such a statement.

For the reader's convenience: the reparametrisation action on the jet variables is
\[
v\mapsto\varphi'v,\qquad a\mapsto\varphi'^2a+\varphi''v,\qquad b\mapsto\varphi'^3b+3\varphi'\varphi''a+\varphi'''v,
\]
and the invariance of $A_{ij;\ell}$ is a direct computation in which the two terms in $\varphi'^3\varphi''$ cancel precisely because of the coefficient $3$.

\subsection{The filtration by the degree in \texorpdfstring{$f'''$}{f'''}}
We spell out, in the notation of this section, the filtration of \S\ref{sec:symbol} for $k=3$. Let $y=\phi(x)$ be a change of coordinates on $X$ and let $J$, $H$, $C$ denote the first, second and third derivatives of $\phi$ at the base point (a linear, a symmetric bilinear and a symmetric trilinear map). The chain rule \eqref{eq:faadibruno} gives, for the jet variables of $\phi\circ f$,
\begin{equation}\label{eq:chainrule}
\tilde v=Jv,\qquad \tilde a=Ja+H(v,v),\qquad \tilde b=Jb+3H(v,a)+C(v,v,v).
\end{equation}
The transition functions of $\E{3}{m}$ are induced by the substitution \eqref{eq:chainrule}, which is affine-linear in $b$ and therefore does not increase the degree in $b$. As in Proposition~\ref{prop:symbol}(1), for every $j\ge0$
\[
F_j\E{3}{m}:=\{P\in\E{3}{m}:\ \deg_bP\le j\}
\]
is a subbundle of $\E{3}{m}$ (its fibres are the same subspace in every trivialisation, so the rank is constant), and $F_0\subset F_1\subset\cdots$ is an increasing filtration by subbundles. In the notation of Demailly \cite[(6.3)]{Dem97} (see also \cite[\S2.1]{Div08}) this is the filtration $F^\bullet_{2}$ written in increasing form: $\deg_b\le j$ if and only if the partial weighted degree of order $2$ is at least $m-3j$.

\begin{proposition}\label{prop:quotient}
Let $q\ge1$. Then $F_{q}\E{3}{5q}=\E{3}{5q}$, and the quotient
\[
Q_q:=\E{3}{5q}\big/F_{q-1}\E{3}{5q}
\]
is a nonzero vector bundle which embeds as a subbundle of $\Sym^{2q}\Om\otimes\Sym^q\Om$. Moreover, for every $r\ge1$ multiplication of jet differentials induces a surjection of vector bundles
\begin{equation}\label{eq:mult}
\Sym^r Q_q\twoheadrightarrow Q_{qr}.
\end{equation}
\end{proposition}

\begin{proof}
The only input specific to order three is the weight count of Step~1; once it is made, the proposition is a direct consequence of the general statements of \S\ref{sec:symbol}, as explained in Step~2, and the proof ends there. We nevertheless add a paragraph which is logically superfluous, in which the general argument is written out explicitly in this simple case, where everything can be followed by hand.

\emph{Step 1: the top $b$-degree part is a polynomial in the symbols.}
Consider a monomial in the generators \eqref{eq:gens} containing $u$ factors $v_i$, $w$ factors $W_{ij}$, $h$ factors $A_{ij;\ell}$ and $d$ factors $D$. Its weighted degree is
\[
u+3w+5h+6d=5(h+d)+u+3w+d,
\]
and its degree in $b$ is at most $h+d$. If the weighted degree is $5q$, then $5(h+d)\le5q$, \textsl{i.e.}\ $h+d\le q$, with equality if and only if $u=w=d=0$, hence $h=q$. Consequently every invariant of weight $5q$ has $\deg_b\le q$, and its homogeneous part of degree exactly $q$ in $b$ is a linear combination of the top $b$-degree parts of monomials $A_{i_1j_1;\ell_1}\cdots A_{i_qj_q;\ell_q}$, that is, of products of $q$ \emph{symbols}
\begin{equation}\label{eq:symbols}
S_{ij;\ell}(v,b):=v_\ell\,(v_ib_j-v_jb_i).
\end{equation}
Each such product is independent of $a$ and has bidegree $(2q,q)$ in $(v,b)$. In the notation of \S\ref{sec:symbol}: $q(3,5q)=q$, the degree $q$ being attained by $A_{12;1}^{\,q}$, which is an invariant of weight $5q$ by Theorem~\ref{thm:rousseau} with $\sigma_q(A_{12;1}^{\,q})=S_{12;1}^{\,q}\neq0$; and the top symbols $\sigma_q(P)=[z^q]P(v,a,zb)$ of the elements of $\E{3}{5q}$ involve $f'$ only, as polynomials of bidegree $(2q,q)$ in $(v,b)$.

\emph{Step 2: conclusion.}
By Proposition~\ref{prop:symbol}(4) one has $F_q\E{3}{5q}=\E{3}{5q}$, and by Corollary~\ref{cor:tensorial} with $s=1$ the top symbol $\sigma_q$ identifies the quotient $Q_q$ with a subbundle of
\[
E_{1,2q}T^*_X\otimes\Sym^q\Om=\Sym^{2q}\Om\otimes\Sym^q\Om ,
\]
which is nonzero since $\sigma_q\neq0$. By Proposition~\ref{prop:symbol}(3), multiplication induces a morphism of bundles
\[
\Sym^rQ_q\longrightarrow\Gr^F_{qr}\E{3}{5qr},\qquad \sigma_{qr}(P_1\cdots P_r)=\sigma_q(P_1)\cdots\sigma_q(P_r),
\]
and since $qr=q(3,5qr)$ the target is the top quotient $Q_{qr}$. The map is surjective because $Q_{qr}$ is spanned, in every trivialisation, by the products of $qr$ symbols \eqref{eq:symbols}, and any such product is a product of $r$ products of $q$ symbols. This proves the proposition.

\emph{The general argument, made explicit in the present case.}
None of what follows is needed; it is the content of Proposition~\ref{prop:symbol} and Corollary~\ref{cor:tensorial} made explicit for $k=3$, $m=5q$, and it shows concretely how the mechanism works. Under \eqref{eq:chainrule}, a term of $P$ of degree $q$ in $b$ contributes to the degree-$q$ part in $\tilde b$ only through the substitution $b\mapsto Jb$, $v\mapsto Jv$, because the summands $3H(v,a)+C(v,v,v)$ of $\tilde b$ lower the degree in $b$, and because the top part contains no $a$. Hence $\sigma_q$ intertwines the transition functions of $\E{3}{5q}$ with those of the tensor bundle $\Sym^{2q}\Om\otimes\Sym^q\Om$ (polynomial functions of bidegree $(2q,q)$ on $T_X\oplus T_X$), and its kernel is $F_{q-1}$.

Its image $Q_q$ is, in every trivialisation, the linear span of the products of $q$ symbols \eqref{eq:symbols}: this is the same subspace of $\Sym^{2q}(\CC^3)^*\otimes\Sym^q(\CC^3)^*$ in every trivialisation, so $\sigma_q$ has constant rank and its image is a subbundle, and this subspace is moreover $\GL_3(\CC)$-stable. Finally, multiplication of polynomials is compatible with \eqref{eq:chainrule}, and the top $b$-degree part of a product is the product of the top $b$-degree parts, which gives the map \eqref{eq:mult}.
\end{proof}

\begin{remark}[Identification of $Q_q$]\label{rem:schur}
On a three-dimensional vector space $V$, the products of the symbols \eqref{eq:symbols} lie in the kernel of the polarisation operator
\[
\delta=\sum_{i=1}^3 v_i\,\frac{\partial}{\partial b_i}\colon\ \Sym^{2q}V^*\otimes\Sym^qV^*\longrightarrow\Sym^{2q+1}V^*\otimes\Sym^{q-1}V^*,
\]
because $\delta S_{ij;\ell}=v_\ell(v_iv_j-v_jv_i)=0$ and $\delta$ is a derivation. The kernel of $\delta$ is the irreducible summand $\Sch{(2q,q,0)}V^*$ of $\Sym^{2q}V^*\otimes\Sym^qV^*$ (it is the summand $j=q$ in the Pieri decomposition $\bigoplus_{j\le q}\Sch{(3q-j,j,0)}V^*$: $\delta$ is $\GL(V)$-equivariant and is nonzero on the highest weight vector $v_1^{2q-j}b_1^{q-j}(v_1b_2-v_2b_1)^j$ of every other summand); the span of the products of symbols is a nonzero $\GL(V)$-stable subspace of it, hence equals it. Thus
\[
Q_q\cong\Sch{(2q,q,0)}\Om,\qquad 0\to F_{q-1}\E{3}{5q}\to\E{3}{5q}\to\Sch{(2q,q,0)}\Om\to0 .
\]
The proof of Theorem~\ref{thm:main} does not use this identification.
\end{remark}

\begin{remark}[Consistency with the known graded structure]\label{rem:rousseau12}
Rousseau's Th\'eor\`eme~1.2 in \cite{Rou06}, restated in \cite[\S1]{Mer10} in the form
\[
\Gr^\bullet\E{3}{m}\ \cong\bigoplus_{u+3w+5h+6d=m}\Sch{(u+w+2h+d,\ w+h+d,\ d)}\Om
\]
(the four indices counting factors $f'$, $W$, $A$, $D$ respectively, as in Step~1 of the proof of Proposition~\ref{prop:quotient}), shows that for $m=5q$ the summand with $h=q$, $u=w=d=0$ is $\Sch{(2q,q,0)}\Om$, that it is the only summand with $h=q$, and that every summand with $d\ge1$ has three rows.

Rousseau's proof, moreover, identifies the highest weight vector of this summand with $(w^1_{12})^q=A_{12;1}^{\,q}$, which is the only highest weight vector of degree $q$ in $f'''$; so the fact that the summand sits at the \emph{top} of the $b$-filtration is implicit in his proof.

What Proposition~\ref{prop:quotient} makes explicit is that this position makes $\Sch{(2q,q,0)}\Om$ a \emph{quotient} of $\E{3}{5q}$, and not merely a subquotient; this is what makes ampleness, which passes to quotients but not to subquotients, sensitive to it.
\end{remark}

\section{Vanishing}\label{sec:vanishing}

We now recall the vanishing theorem which will be played against the quotients of the previous section. It is a theorem of Br\"uckmann and Rackwitz on tensor forms on complete intersections, which we state in the language of Schur functors, and we deduce from it the vanishing of the global sections of the bundles we need.

With the notation for partitions fixed in \S\ref{sec:statement}, we use the following vanishing theorem.

\begin{theorem}[{Br\"uckmann--Rackwitz, \cite[Introduction and Thm.~4(iii), Cor.~(11)]{BR90}; see also \cite[Thm.~6]{Div08} and \cite[\S1]{Ete21}}]\label{thm:BR}
Let $Z\subset\PP^N$ be a smooth complete intersection of $c$ hypersurfaces of degrees $\ge2$, and let $\lambda$ be a nonempty partition. If
\[
\sum_{j=1}^{c}\lambda'_j<N-c=\dim Z,
\]
then $H^0(Z,\Sch{\lambda}\OmZ)=0$.
\end{theorem}

In words: $Z$ carries no nonzero tensor forms of symmetry type $\lambda$ if the Young diagram of $\lambda$ has fewer than $\dim Z$ cells in its first $\operatorname{codim}Z$ columns (in \cite{BR90} the statement is phrased in terms of a Young tableau $T$ and of the sheaf $\mathcal T^T$ of $T$-symmetrical tensor forms, which is $\Sch{\lambda}\OmZ$ for the diagram $\lambda$ of $T$).

This is the vanishing theorem underlying our earlier result that smooth complete intersections carry no invariant jet differentials of order less than $\dim Z/\operatorname{codim}Z$ \cite[Thms.~1 and~7]{Div08}, and Etesse's observation \cite[\S1]{Ete21} that $\Sch{\lambda}\OmZ$ cannot be ample when $\lambda$ has $\ell$ rows and $(1+\ell)c<N$.

\begin{lemma}\label{lem:vanishing}
Let $Z\subset\PP^N$ be a smooth complete intersection of $c$ hypersurfaces of degrees $\ge2$ with $3c<N$, or let $Z=X\subset\PP^4$ be any smooth hypersurface. Then for every integer $t\ge1$
\[
H^0\big(Z,\Sym^{2t}\OmZ\otimes\Sym^t\OmZ\big)=0 .
\]
\end{lemma}

\begin{proof}
By Pieri's rule,
\[
\Sym^{2t}\OmZ\otimes\Sym^t\OmZ\ \cong\ \bigoplus_{j=0}^{t}\Sch{(3t-j,\,j,\,0,\dots,0)}\OmZ .
\]
Every summand corresponds to a nonempty partition with at most two rows, so $\sum_{j\le c}\lambda'_j\le2c<N-c$, and Theorem~\ref{thm:BR} applies. For $X\subset\PP^4$ of degree $d\ge2$ this is the case $c=1$, $N=4$. If $d=1$, then $X\cong\PP^3$ and the dual Euler sequence
\[
0\to\Omega^1_{\PP^3}\to\cO_{\PP^3}(-1)^{\oplus4}\to\cO_{\PP^3}\to0
\]
exhibits $\Sym^{2t}\Om\otimes\Sym^t\Om$ as a subsheaf of a direct sum of copies of $\cO_{\PP^3}(-3t)$, which has no nonzero sections.
\end{proof}

\section{Non-ampleness on hypersurfaces of \texorpdfstring{$\PP^4$}{P4}}\label{sec:nonample}

With the quotient bundles of \S\ref{sec:quotient} and the vanishing of \S\ref{sec:vanishing} at hand, the proof of the main theorem is a few lines: ampleness descends to quotients, ampleness gives sections of symmetric powers, and there are no sections to be had.

\begin{proof}[Proof of Theorem~\ref{thm:main}]
For $d\ge2$ this is Theorem~\ref{thm:criterion} with $k=3$, $s=1$, $c=1$, $N=4$; we spell out the argument, which through Lemma~\ref{lem:vanishing} also covers $d=1$. Suppose that $\E{3}{5q}$ is ample. By Proposition~\ref{prop:quotient} it has the nonzero quotient bundle $Q_q$, which is then ample by (H1). By (H2), $\Sym^rQ_q$ is generated by global sections for all $r\gg0$; fix such an $r$. By the surjection \eqref{eq:mult} and (H3), the nonzero bundle $Q_{qr}$ is generated by global sections, hence $H^0(X,Q_{qr})\neq0$. Since $Q_{qr}\subset\Sym^{2qr}\Om\otimes\Sym^{qr}\Om$, this contradicts Lemma~\ref{lem:vanishing} with $t=qr$.
\end{proof}

\begin{remark}
Once $Q_q\cong\Sch{(2q,q,0)}\Om$ is known (Remark~\ref{rem:schur}), Theorem~\ref{thm:main} also follows from (H1) and from Etesse's observation \cite[\S1]{Ete21} that a Schur power with two rows cannot be ample on a hypersurface of $\PP^4$, since $(1+2)\cdot1<4$. Etesse deduces the global generation of $\Sch{m\lambda}\OmZ$ from the ampleness of $\Sch{\lambda}\OmZ$ via the theorem of Laytimi--Nahm \cite[Prop.~2.5 and Rem.~2.6]{Ete21}; the surjection \eqref{eq:mult} is the elementary substitute for this step in our situation.
\end{remark}

\section{Arbitrary dimension and codimension}\label{sec:general}

Nothing in \S\S\ref{sec:quotient}--\ref{sec:nonample} is specific to threefolds except the use of Rousseau's theorem, which is stated in dimension $3$. We now observe that Rousseau's own proof gives the generators in every dimension, and we draw the consequences for complete intersections of arbitrary dimension and codimension.

\subsection{Generators of the third-jet algebra in every dimension}
The proof of Theorem~\ref{thm:rousseau} in \cite[\S3]{Rou06} goes as follows. Let $V=\CC^3$ be the space of $3$-jets of one scalar function, on which the group $G'_3$ of reparametrisations tangent to the identity acts linearly and unimodularly; a $3$-jet of a curve in $\CC^n$ is a system of $n$ vectors $u_1,\dots,u_n$ of $V$, namely $u_i=(f_i',f_i'',f_i''')$, and $E_{3,m}$ in dimension $n$ is the weight-$m$ part of the algebra of $G'_3$-invariant polynomials in $u_1,\dots,u_n$. Rousseau shows that the three multilinear forms
\[
F_1(u_1)=f_1',\qquad F_2(u_1,u_2)=W_{12},\qquad F_3(u_1,u_2,u_3)=A_{12;3}
\]
form a \emph{complete system} of $G'_3$-invariants of a system of two vectors, and then invokes Popov's theorem \cite[Th\'eor\`eme~2.6(2)]{Rou06}. The terminology is that of Popov and Vinberg, as recalled in \cite[\S2.3]{Rou06}: given a multilinear form $F(u_1,\dots,u_l)$ on $V$ and vectors $x_1,\dots,x_m\in V$, the \emph{substitution space} of $F$ is the vector space spanned by all the polynomials $F(x_{i_1},\dots,x_{i_l})$, $1\le i_1,\dots,i_l\le m$, repetitions allowed \cite[D\'ef.~2.4]{Rou06}; a family of $G$-invariant multilinear forms is a \emph{complete system} of $G$-invariants of a system of $m$ vectors if the substitution spaces of its members generate the algebra of $G$-invariant polynomials in $x_1,\dots,x_m$ \cite[D\'ef.~2.5]{Rou06}. Popov's theorem says: since $G'_3\subset\SL(V)$ and $\dim V=3$, a complete system for $\dim V-1=2$ vectors, together with the determinant, is a complete system for \emph{any} number of vectors. Rousseau states his theorem for $n=3$, but the theorem he invokes gives it for every $n$, as he remarks himself \cite[Remarque~3.5~1)]{Rou06}. We record this explicitly.

\begin{proposition}[{Rousseau--Popov}]\label{prop:alln}
For every $n\ge2$, the algebra of invariant $3$-jet differentials in dimension $n$ is generated by the polynomials
\[
v_i\ (1\le i\le n),\qquad W_{ij}\ (i<j),\qquad A_{ij;\ell}\ (i<j,\ 1\le\ell\le n),
\]
\[
D_{ijk}=\det(v_{\{ijk\}},a_{\{ijk\}},b_{\{ijk\}})\ (i<j<k),
\]
the last family being the $3\times3$ minors of the $3\times n$ matrix with rows $v,a,b$. Their weights are $1,3,5,6$ and their degrees in $b$ are $0,0,1,1$.
\end{proposition}

\begin{proof}
The polynomials listed are exactly the values of $F_1,F_2,F_3$ and of $\det$ on all choices of arguments among $u_1,\dots,u_n$ (repetitions allowed): they span the substitution spaces of the complete system $\{F_1,F_2,F_3,\det\}$ in the sense of \cite[D\'ef.~2.4 and~2.5]{Rou06}. By \cite[Th\'eor\`eme~2.6(2)]{Rou06} they generate the algebra of $G'_3$-invariants of $n$ vectors for every $n$. Passing from $G'_3$-invariants to invariants under the full reparametrisation group only amounts to taking weighted-homogeneous components, and all the generators are weighted-homogeneous. (For $n=2$ the minors are absent and one recovers the description of $E_{3,\bullet}$ in dimension $2$ of \cite[Th\'eor\`eme~3.2]{Rou06}, \textsl{cf.}\ \cite[\S4]{Mer10}.)
\end{proof}

The same reasoning with $k=2$ gives the classical fact that the invariant $2$-jet algebra in every dimension is generated by the $v_i$ and the $W_{ij}$ \cite[\S12]{Dem97} (see also \cite[\S4]{Mer10}), which we use in Remark~\ref{rem:ordertwo}.

\subsection{The extremal quotient in dimension \texorpdfstring{$n$}{n}}
\begin{proposition}\label{prop:quotientn}
Let $Z$ be a smooth variety of dimension $n\ge2$ and $q\ge1$. Then $\Ez{3}{5q}$ has a nonzero quotient bundle $Q_q\subset\Sym^{2q}\OmZ\otimes\Sym^q\OmZ$, isomorphic to $\Sch{(2q,q,0,\dots,0)}\OmZ$, and multiplication induces surjections $\Sym^rQ_q\twoheadrightarrow Q_{qr}$ for all $r\ge1$.
\end{proposition}

\begin{proof}
Identical to the proofs of Proposition~\ref{prop:quotient} and Remark~\ref{rem:schur}, with Proposition~\ref{prop:alln} in place of Theorem~\ref{thm:rousseau}: the minors $D_{ijk}$ have weight $6$ and $b$-degree $1$ exactly like $D$, so the weight count of Step~1 is unchanged, and the kernel of $\delta=\sum_{i=1}^n v_i\,\partial/\partial b_i$ on $\Sym^{2q}V^*\otimes\Sym^qV^*$ is the irreducible summand $\Sch{(2q,q,0,\dots,0)}V^*$ for every $n\ge2$.
\end{proof}

\begin{proof}[Proof of Theorem~\ref{thm:ci}]
Note that $3c<N$ forces $\dim Z=N-c>2c\ge2$. This is Theorem~\ref{thm:criterion} with $k=3$ and $s=1$: by Proposition~\ref{prop:alln} the weight count of Step~1 in the proof of Proposition~\ref{prop:quotient} holds in every dimension, so the top symbols of the elements of $\Ez{3}{5q}$ involve $f'$ only, and $(s+2)c=3c<N$. Explicitly: if $\Ez{3}{5q}$ were ample, by Proposition~\ref{prop:quotientn} and (H1)--(H3) one would have $H^0(Z,Q_{qr})\ne0$ for some $r\ge1$, contradicting Lemma~\ref{lem:vanishing}.
\end{proof}

\begin{corollary}\label{cor:conj}
Readings \textup{(A)} and \textup{(A$'$)} of Conjectures~\ref{conj:DT} and~\ref{conj:Deng} fail at $k=3$ for every $(N,c)$ with $3c<N\le4c$. More precisely, for every smooth complete intersection $Z\subset\PP^N$ of codimension $c$ with $3c<N$ and multidegree $\ge2$, and every integer $M\ge1$, the bundle $\Ez{3}{m}$ fails to be ample for infinitely many multiples $m$ of $M$; and $k=3$ is within the range $k\ge N/c-1$ of the conjecture exactly when $N\le4c$.
\end{corollary}

\begin{proof}
Theorem~\ref{thm:ci} applies to $m\in5\ZZ_{>0}$, in particular to all multiples of $\operatorname{lcm}(5,M)$. Since it holds for \emph{every} smooth complete intersection, no genericity assumption can restore ampleness.
\end{proof}

\begin{remark}
The smallest cases are $(N,c)=(4,1)$ (Theorem~\ref{thm:main}), $(7,2)$ and $(8,2)$ (complete intersections of two hypersurfaces in $\PP^7$ or $\PP^8$, of dimensions $5$ and $6$), and $(10,3)$, $(11,3)$, $(12,3)$. Hypersurfaces in $\PP^N$ with $N\ge5$ are not covered by order three, because there the conjecture requires $k\ge N-1\ge4$: hypersurfaces in $\PP^5$ are treated in \S\ref{sec:k4} (Theorem~\ref{thm:k4}), while $N\ge6$ would require order $k\ge5$, see Remark~\ref{rem:general}.
\end{remark}

\section{Order four}\label{sec:k4}

The threshold of the conjecture for hypersurfaces in $\PP^N$ is $k\ge N-1$, so the next case after $\PP^4$ concerns jet differentials of order $4$ on fourfolds in $\PP^5$. The invariant $4$-jet algebra is no longer described by a handful of generators, but Merker has computed it in dimension $4$, and it turns out that his list contains exactly the information needed to run the argument of the previous sections once more.

The conclusion is the same, and for the same reason: the generators which achieve the extremal degree in $f''''$ never involve all four derivatives, so that Theorem~\ref{thm:criterion} applies with $s=2$. Here the top symbols do involve $f''$, so the top quotient is not naturally contained in a tensor bundle, and the argument has to be run one level down, inside the $2$-jet differentials with values in $\Sym^q\OmZ$.

Throughout this section $Z$ is a smooth variety of dimension $n\ge4$, and in local coordinates we write
\[
v=f'(0),\quad a=f''(0),\quad b=f'''(0),\quad d=f''''(0)\qquad(v,a,b,d\in\CC^n).
\]
For a $3\times3$ minor we write $\det_{ijk}(x,y,z)$ for the determinant of the rows $i,j,k$ of the $n\times3$ matrix $(x\,|\,y\,|\,z)$.

\subsection{Merker's generators}
Merker \cite[\S11]{Mer10} proves that in dimension $n=4$ the algebra of invariant $4$-jet differentials is generated by the polarisations (the $\GL_4$-translates) of sixteen \emph{bi-invariants}. In Merker's terminology \cite[\S7]{Mer10}, a bi-invariant is a polynomial in the jet variables which is invariant under reparametrisation and, in addition, under the action of the unipotent group of upper triangular matrices of $\GL_n(\CC)$ acting on the components $f_1,\dots,f_n$; in other words, the bi-invariants are the vectors fixed by the unipotent group, that is, the linear combinations of highest weight vectors of the $\GL_n(\CC)$-representations $E_{k,m}$ (Merker's generators, being multihomogeneous, are highest weight vectors), and the algebra of invariants is generated by the $\GL_n(\CC)$-translates of any set of generators of the algebra of bi-invariants. The sixteen bi-invariants are denoted
\begin{gather*}
f_1',\ \Lambda^3,\ \Lambda^5,\ \Lambda^7,\ D^6,\ D^8,\ N^{10},\ W^{10},\\
 M^8,\ E^{10},\ L^{12},\ Q^{14},\ R^{15},\ U^{17},\ V^{19},\ X^{21},
\end{gather*}
the superscript being the weight; the polarisations number $2835$. The seven initial ones (all but the trivial $f_1'$) are given explicitly \cite[\S11, ``First loop'']{Mer10}: with $\Delta^{p,q}_{12}=f_1^{(p)}f_2^{(q)}-f_2^{(p)}f_1^{(q)}$,
\begin{align*}
\Lambda^3&=W_{12},\qquad \Lambda^5=A_{12;1},\\ D^6&=\det\nolimits_{123}(v,a,b),\qquad W^{10}=\det(v,a,b,d),\\
\Lambda^7&=\Delta^{1,4}_{12}\,v_1v_1+(\text{terms without }d),\\
D^8&=v_1\det\nolimits_{123}(v,a,d)-6\,a_1\det\nolimits_{123}(v,a,b),\\
N^{10}&=v_1v_1\det\nolimits_{123}(v,b,d)-3\,v_1a_1\det\nolimits_{123}(v,a,d)\\&\qquad+(\text{terms without }d),
\end{align*}
and the remaining eight are defined through the reduced syzygies of \cite[\S11]{Mer10}, which are the following exact polynomial identities:
\begin{gather*}
f_1'f_1'M^8=-5\Lambda^5\Lambda^5+3\Lambda^3\Lambda^7,\qquad f_1'E^{10}=-6\Lambda^5D^6+3\Lambda^3D^8,\\
f_1'L^{12}=-\Lambda^7D^6+5\Lambda^3N^{10},
\end{gather*}
\begin{equation}\label{eq:QRX}
\begin{gathered}
f_1'Q^{14}=\Lambda^7D^8-10\Lambda^5N^{10},\qquad
f_1'R^{15}=D^8D^8-12D^6N^{10},\\
f_1'X^{21}=4D^8Q^{14}-5\Lambda^7R^{15},
\end{gathered}
\end{equation}
and
\[
f_1'U^{17}=4D^8E^{10}+3\Lambda^3R^{15},\qquad f_1'V^{19}=8N^{10}E^{10}+\Lambda^5R^{15}.
\]

The coefficient $6$ in $D^8$ is forced by invariance: under $t\mapsto t+\varepsilon t^2/2$ one has, to first order in $\varepsilon$,
\[
a\mapsto a+\varepsilon v,\qquad \det\nolimits_{123}(v,a,d)\mapsto\det\nolimits_{123}(v,a,d)+6\varepsilon\det\nolimits_{123}(v,a,b);
\]
it is also what the bracket definition $D^8=[D^6,f_1']$ of \cite[\S7]{Mer10} gives, and it is the coefficient appearing in the general formula for $D^8_{[i,j,k];\alpha}$ in \cite[\S12]{Mer10}, whereas the first-loop display in \cite[\S11]{Mer10} has the coefficient $3$ there, which does not affect the leading part in $d$. Only the shape of these identities, not the numerical constants, is used below, except in the computation of $X^{21}_{\rm top}$.

Reading off the degree in $d$ from these formulas gives the following table of upper bounds (which are the exact degrees for the generators we shall use):
\[
\begin{array}{c|cccccccc}
&f_1'&\Lambda^3&\Lambda^5&\Lambda^7&D^6&D^8&N^{10}&W^{10}\\\hline
\text{weight}&1&3&5&7&6&8&10&10\\
\deg_d\le&0&0&0&1&0&1&1&1
\end{array}
\]
\[
\begin{array}{c|cccccccc}
&M^8&E^{10}&L^{12}&Q^{14}&R^{15}&U^{17}&V^{19}&X^{21}\\\hline
\text{weight}&8&10&12&14&15&17&19&21\\
\deg_d\le&1&1&1&2&2&2&2&3
\end{array}
\]
In every case $\deg_d\le\operatorname{weight}/7$, with equality only for $\Lambda^7$, $Q^{14}$, $X^{21}$. It is convenient to record the \emph{defect}
\[
\epsilon(P):=\operatorname{weight}(P)-7\deg_dP\ \ge0
\]
of each generator:
\begin{gather*}
\epsilon=0\ \text{ for }\ \Lambda^7,\ Q^{14},\ X^{21};\qquad
\epsilon=1\ \text{ for }\ f_1',\ D^8,\ M^8,\ R^{15};\\
\epsilon\ge3\ \text{ for all the others, for instance }\ \epsilon(W^{10})=3.
\end{gather*} The defect is additive on products.

\begin{proposition}[{Merker--Popov}]\label{prop:alln4}
For every $n\ge4$, the algebra of invariant $4$-jet differentials in dimension $n$ is generated by the polarisations of the sixteen bi-invariants above. Consequently every invariant $P$ of weight $m=7q+r$, $0\le r\le6$, in dimension $n\ge4$ satisfies $\deg_dP\le q$.

If moreover $r\in\{0,1,2\}$, the homogeneous part of degree exactly $q$ in $d$ of $P$ is a linear combination of products of $r$ leading parts of defect-one generators with leading parts of defect-zero generators, all taken up to polarisation. The leading parts of the defect-one generators are
\begin{equation}\label{eq:lead4b}
\begin{gathered}
v_1,\qquad D^8_{\rm top}=v_1\,\det\nolimits_{123}(v,a,d),\\
M^8_{\rm top}=3\,W_{12}\,(v_1d_2-v_2d_1),\qquad R^{15}_{\rm top}=v_1\,\det\nolimits_{123}(v,a,d)^2,
\end{gathered}
\end{equation}
and those of the defect-zero generators are
\begin{equation}\label{eq:lead4}
\begin{aligned}
\Lambda^7_{\rm top}&=v_1^2\,(v_1d_2-v_2d_1),\\
Q^{14}_{\rm top}&=v_1^2\,(v_1d_2-v_2d_1)\,\det\nolimits_{123}(v,a,d),\\
X^{21}_{\rm top}&=-\,v_1^2\,(v_1d_2-v_2d_1)\,\det\nolimits_{123}(v,a,d)^2 .
\end{aligned}
\end{equation}
\end{proposition}

\begin{proof}
Merker's generators are polynomials in the $4$-jets $u_1,\dots,u_4\in V=\CC^4$ of the components of $f$; replacing each of them by its complete polarisation, a $G'_4$-invariant multilinear form whose substitution space contains it \cite[\S2.3 and D\'ef.~2.4]{Rou06}, one obtains a complete system of invariants (in the sense of \cite[D\'ef.~2.5]{Rou06}) of a system of $4=\dim V$ vectors of $V=\CC^4$, the space of $4$-jets of one scalar function, for the unipotent group $G'_4$; by Popov's theorem \cite[Th\'eor\`eme~2.6(1)]{Rou06} it is a complete system for any number of vectors, \textsl{i.e.}\ in every dimension $n\ge4$.

Polarisation does not change the weight nor the degree in $d$, so the table gives $\deg_dP\le m/7$, \textsl{i.e.}\ $\deg_dP\le q$, for every invariant $P$ of weight $m=7q+r$. A monomial in the generators of weight $m$ and $d$-degree $q$ has total defect $r$; since the defect is additive and every generator has defect $0$, $1$ or $\ge3$, for $r\le2$ such a monomial is a product of $r$ (polarised) defect-one generators and of (polarised) defect-zero generators, each with its extremal $d$-degree, and its leading $d$-part is the product of the leading parts of the factors.

The leading parts \eqref{eq:lead4b} are read off from $f_1'$, from $D^8$, and from the identities
\[
(f_1')^2M^8=-5\Lambda^5\Lambda^5+3\Lambda^3\Lambda^7,\qquad f_1'R^{15}=D^8D^8-12D^6N^{10},
\]
in which only $\Lambda^7$ contains $d$, respectively only the first term reaches $d$-degree $2$. The leading parts \eqref{eq:lead4} are computed from \eqref{eq:QRX} by taking the parts of $d$-degree $2$, $2$ and $3$: since $\deg_dN^{10}\le1$ and $\deg_dD^6=0$, one gets, with $D^8_{\rm top}=v_1\det_{123}(v,a,d)$,
\[
f_1'Q^{14}_{\rm top}=\Lambda^7_{\rm top}D^8_{\rm top},\qquad f_1'R^{15}_{\rm top}=(D^8_{\rm top})^2,
\]
and finally
\[
f_1'X^{21}_{\rm top}=4D^8_{\rm top}Q^{14}_{\rm top}-5\Lambda^7_{\rm top}R^{15}_{\rm top}=(4-5)\,v_1\Lambda^7_{\rm top}\det\nolimits_{123}(v,a,d)^2 ,
\]
which gives \eqref{eq:lead4}.
\end{proof}

The essential feature of \eqref{eq:lead4b} and \eqref{eq:lead4} is that \emph{none of the leading parts involves $b=f'''$}: they are polynomials in the three vectors $v,a,d$ only. This is in contrast with the Wronskian $W^{10}=\det(v,a,b,d)$, which involves all four derivatives but has $\deg_d/\operatorname{weight}=1/10<1/7$, \textsl{i.e.}\ defect $3$, and therefore does not contribute to the extremal graded piece at the weights $m\equiv0,1,2\pmod7$.

\subsection{The extremal quotient and its irreducible summands}\label{sec:G4}
Proposition~\ref{prop:alln4} says that for $m=7q+r$, $r\in\{0,1,2\}$, the hypothesis of Theorem~\ref{thm:criterion} holds for $\Ez{4}{m}$ with $s=2$: $q(4,m)=q$, and the top symbols involve only $f'$ and $f''$. Theorem~\ref{thm:k4} is therefore an instance of Theorem~\ref{thm:criterion}, with $(s+2)c=4c<N$. As in \S\ref{sec:quotient}, we spell out what the general argument amounts to here; the difference with order three is that the top symbols now involve $f''$, so that the top quotient lives in a bundle of $2$-jet differentials rather than in a tensor bundle, and one more filtration is needed before the Br\"uckmann--Rackwitz theorem can be applied.

The chain rule \eqref{eq:faadibruno} for $y=\phi(x)$ reads $\tilde d=Jd+(\text{terms in }v,a,b)$, so the degree in $d$ defines the increasing filtration $F_\bullet$ of $\Ez{4}{m}$ by subbundles of \S\ref{sec:symbol}, and for $m=7q+r$ Proposition~\ref{prop:alln4} gives $F_q=\Ez{4}{m}$. For $m=7q+r$ with $q\ge1$ and $r\in\{0,1,2\}$ let
\[
G_m:=\Ez{4}{m}\big/F_{q-1}\Ez{4}{m}.
\]
For every $p\ge0$ and $t\ge0$ we also write $B_{p,t}:=E^{GG}_{2,p}T^*_Z\otimes\Sym^t\OmZ$, with $E^{GG}_{2,p}T^*_Z$ the Green--Griffiths bundle of \S\ref{sec:symbol}; note that $B_{p,t}\otimes B_{p',t'}\to B_{p+p',t+t'}$ by multiplication of polynomials.

\begin{proposition}\label{prop:G4}
Let $n\ge4$, $q\ge1$, $r\in\{0,1,2\}$ and $m=7q+r$.
\begin{enumerate}
\item $G_m$ is a nonzero vector bundle, and the map $P\mapsto[z^q]P(v,a,b,zd)$ (the coefficient of $z^q$, \textsl{i.e.}\ the homogeneous part of degree $q$ in $d$, as in the proof of Proposition~\ref{prop:quotient}) identifies it with a subbundle of $B_{3q+r,q}=E^{GG}_{2,3q+r}T^*_Z\otimes\Sym^q\OmZ$.
\item For every $(p,t)\ne(0,0)$ the bundle $B_{p,t}$ carries a filtration by subbundles whose graded pieces are direct sums of Schur bundles $\Sch{\lambda}\OmZ$ with $\lambda$ nonempty and having at most three rows; the same holds for $G_m$.
\item For every $\ell\ge1$, multiplication induces a bundle map $\Sym^\ell G_m\to B_{\ell(3q+r),\ell q}$ which is nonzero: it sends $\big(v_1^{\,r}(\Lambda^7_{\rm top})^q\big)^{\otimes\ell}$ to $v_1^{\,\ell r}(\Lambda^7_{\rm top})^{\ell q}\ne0$.
\end{enumerate}
\end{proposition}

\begin{proof}
(1) The top $d$-degree part of $P(\tilde v,\tilde a,\tilde b,\tilde d)$ is $P_{\rm top}(\tilde v,\tilde a,\tilde b,Jd)$, where $(\tilde v,\tilde a,\tilde b)$ is the full $3$-jet transform of $(v,a,b)$; hence $\sigma_q(P):=[z^q]P(v,a,b,zd)$ is a bundle map from $\Ez{4}{m}$ to the bundle of polynomials of weight $m-4q=3q+r$ in $(v,a,b)$, transforming as Green--Griffiths $3$-jet differentials, with values in $\Sym^q\OmZ$, \textsl{i.e.}\ to $E^{GG}_{3,3q+r}T^*_Z\otimes\Sym^q\OmZ$; its kernel is $F_{q-1}$.

By Proposition~\ref{prop:alln4} its image is spanned, in every trivialisation, by products of polarisations of the polynomials \eqref{eq:lead4b} and \eqref{eq:lead4}, which do not involve $b$. Polynomials not involving $b$ form the subbundle $E^{GG}_{2,3q+r}T^*_Z\subset E^{GG}_{3,3q+r}T^*_Z$ (the substitution $(v,a)\mapsto(\tilde v,\tilde a)$ does not introduce $b$). The image is the same $\GL_n$-stable subspace in every trivialisation, so it has constant rank and is a subbundle; it is nonzero since it contains $v_1^{\,r}(\Lambda^7_{\rm top})^q$, the leading part of the invariant $(f_1')^r(\Lambda^7)^q$.

(2) The substitution $\tilde a=Ja+H(v,v)$ does not increase the degree in $a$, so $E^{GG}_{2,p}T^*_Z$ is filtered by the degree in $a$, with graded pieces $\Sym^{p-2j}\OmZ\otimes\Sym^j\OmZ$ (polynomials of degree $j$ in $a$ and $p-2j$ in $v$). Hence $B_{p,t}$ has a filtration with graded pieces $\Sym^{p-2j}\OmZ\otimes\Sym^j\OmZ\otimes\Sym^t\OmZ$, and by Pieri's rule \cite[Prop.~1]{Div08} a tensor product of three symmetric powers only contains $\Sch{\lambda}$ with $\lambda_i=0$ for $i>3$; the partitions are nonempty because $(p,t)\ne(0,0)$.

The induced filtration on the subbundle $G_m\subset B_{3q+r,q}$ has graded pieces which are subbundles of these tensor bundles with fibre a fixed $\GL_n$-stable subspace, as in the proof of Corollary~\ref{cor:tensorial}; by Lemma~\ref{lem:equivariant} they are direct sums of Schur bundles $\Sch{\lambda}\OmZ$ occurring in them.

(3) Multiplication of polynomials in $(v,a,d)$ is compatible with the transition functions of the bundles $B_{p,t}$, and the product of $\ell$ copies of the nonzero polynomial $v_1^{\,r}(\Lambda^7_{\rm top})^q$ is nonzero. (For $r=0$ the image lies in $G_{7q\ell}\subset B_{3q\ell,q\ell}$; for $r>0$ it need not lie in the extremal quotient of $\Ez{4}{\ell m}$, which is why we work in the ambient bundles $B_{p,t}$.)
\end{proof}

\begin{lemma}\label{lem:vanishing4}
Let $Z\subset\PP^N$ be a smooth complete intersection of $c$ hypersurfaces of degrees $\ge2$ with $4c<N$. Then $H^0(Z,B_{p,t})=0$ for every $(p,t)\ne(0,0)$; in particular $H^0(Z,G_m)=0$ for every $m\ge7$ with $m\equiv0,1,2\pmod 7$.
\end{lemma}

\begin{proof}
A nonempty partition with at most three rows satisfies $\sum_{j\le c}\lambda'_j\le3c<N-c$, so by Theorem~\ref{thm:BR} every graded piece of the filtration of Proposition~\ref{prop:G4}(2) has no nonzero global sections, and the filtration lemma \cite[Lemma~2]{Div08} gives $H^0(Z,B_{p,t})=0$; the assertion on $G_m$ follows since $G_m\subset B_{3q+r,q}$.
\end{proof}

\begin{proof}[Proof of Theorem~\ref{thm:k4}]
Note that $4c<N$ forces $n=N-c>3c\ge3$, so $n\ge4$ and \S\ref{sec:k4} applies. As explained at the beginning of \S\ref{sec:G4}, this is Theorem~\ref{thm:criterion} with $k=4$ and $s=2$, the hypothesis being verified by Proposition~\ref{prop:alln4}. Explicitly: write $m=7q+r$ with $q\ge1$, $r\in\{0,1,2\}$, and suppose $\Ez{4}{m}$ ample. Then $G_m$ is ample by (H1), $\Sym^\ell G_m$ is globally generated for some $\ell\ge1$ by (H2), and so is its image in $B_{\ell(3q+r),\ell q}$ under the multiplication map of Proposition~\ref{prop:G4}(3), by (H3). This image is a nonzero subsheaf of $B_{\ell(3q+r),\ell q}$, hence has a nonzero global section, contradicting Lemma~\ref{lem:vanishing4}.
\end{proof}

\begin{corollary}\label{cor:conj4}
Readings \textup{(A)} and \textup{(A$'$)} of Conjectures~\ref{conj:DT} and~\ref{conj:Deng} fail at $k=4$ for every $(N,c)$ with $4c<N\le5c$, in particular for hypersurfaces in $\PP^5$; there, $\Ez{4}{m}$ is not ample at any weight $m\ge7$ with $m\equiv0,1,2\pmod7$, so that reading \textup{(B)} could only hold at the remaining weights.
\end{corollary}

\begin{proof}
As for Corollary~\ref{cor:conj}, using the weights $m\ge7$ with $m\equiv0,1,2\pmod7$, which include all large multiples of any $M$.
\end{proof}

\begin{corollary}[Every dimension]\label{cor:alldim}
Let $n\ge3$ and let $c$ be an integer with $n/4\le c<n/2$. Then for every smooth complete intersection $Z\subset\PP^{n+c}$ of codimension $c$ and multidegree $\ge2$, readings \textup{(A)} and \textup{(A$'$)} of Conjectures~\ref{conj:DT} and~\ref{conj:Deng} fail at the threshold order $k=\lceil n/c\rceil$ of the conjecture itself. Such a $c$ exists for every $n\ge3$.
\end{corollary}

\begin{proof}
Write $N=n+c$; the threshold of the conjecture is $k\ge N/c-1=n/c$. If $n/3\le c<n/2$, then $3c<N\le4c$, the threshold is $k=3$, and Corollary~\ref{cor:conj} applies; if $n/4\le c<n/3$, then $4c<N\le5c$, the threshold is $k=4$, and Corollary~\ref{cor:conj4} applies. The interval $[n/4,n/2)$ contains $c=1$ for $n=3$ and has length $n/4\ge1$ for $n\ge4$, so it always contains an integer.
\end{proof}

\begin{remark}
The corollary leaves out the small codimensions $c<n/4$, where the threshold is $k\ge5$ (for instance hypersurfaces of dimension $n\ge5$, see Remark~\ref{rem:general}), and the large codimensions $c\ge n/2$, where the threshold is $k\le2$: for $k=1$ the conjecture is the theorem of Brotbek--Darondeau and Xie, and for $k=2$ the obstruction of Remark~\ref{rem:ordertwo} is exactly complementary to the threshold. In each dimension the smallest instances are: $n=3$, $c=1$, order $3$; $n=4$, $c=1$, order $4$, the only dimension for which order $4$ is indispensable; $n=5,6$, $c=2$, order $3$; $n=7,8$, $c=2$ at order $4$ or $c=3$ at order $3$.
\end{remark}

\begin{remark}[Threefolds and order four]
On a smooth threefold $X\subset\PP^4$ the same construction gives the extremal quotient $G_7=\Sch{(3,1,0)}\Om$ of $\E{4}{7}$ (two rows: $\Lambda^7_{\rm top}$ only involves $v$ and $d$), so $\E{4}{7}$ is not ample on any smooth hypersurface of $\PP^4$ of degree $\ge2$ (this is Theorem~\ref{thm:criterion} with $k=4$, $s=1$, $c=1$, $N=4$: the products of $\ell$ leading parts of polarisations of $\Lambda^7$ lie in $\Sym^{3\ell}\Om\otimes\Sym^\ell\Om$, whose irreducible summands have two rows), using Merker's description in dimension $3$ \cite[\S11]{Mer10} or the restriction of the dimension-$4$ generators.

For $q\ge2$, however, $G_{7q}$ contains the three-row leading part of $Q^{14}$, and Theorem~\ref{thm:BR} on $\PP^4$ requires at most two rows: the method gives no information on $\E{4}{7q}$ for $q\ge2$ on threefolds. (On threefolds the conjecture predicts ampleness for every $k\ge3$, so order $4$ is within its range, but $k=3$ is the threshold order there and is already settled by Theorem~\ref{thm:main}.)
\end{remark}

\begin{remark}[Why $k-1$ rows]\label{rem:whyrows}
In orders $3$ and $4$ the generators achieving the extremal ratio $\deg_{f^{(k)}}/\operatorname{weight}=1/(2k-1)$ have leading parts which involve only $k-1$ of the $k$ vectors $f',\dots,f^{(k)}$: $(v,b)$ for $k=3$, $(v,a,d)$ for $k=4$. By Pieri's rule their span has irreducible summands with at most $k-1$ rows, which have no sections on hypersurfaces of dimension $k$, the threshold dimension for order $k$.

The summands with $k$ rows, which carry the global sections at the threshold \cite{Div08,Rou06}, come from the Wronskian $\det(f',\dots,f^{(k)})$ (for $k=3$ by \cite[Th\'eor\`eme~1.2]{Rou06}, for $k=4$ because all of Merker's generators except $W^{10}$ involve only three of the components $f_1,\dots,f_4$ \cite[\S11]{Mer10}; for $k\ge5$ this is not known), and its ratio $2/(k(k+1))$ is smaller than $1/(2k-1)$ as soon as $k\ge3$, so that it never reaches the top.

For $k=2$ the Wronskian is $W_{ij}$ itself, the two ratios coincide, and the extremal quotient has the full number $2$ of rows of the threshold dimension $2$ (on surfaces it is $K_X^{\,w}$, Remark~\ref{rem:dim2}); this is why order $2$ is the only order at which the obstruction is exactly complementary to the threshold of the conjecture (Remark~\ref{rem:ordertwo}). Whether the pattern persists for $k\ge5$ depends on the still unknown structure of the invariant algebras in dimension $\ge3$; see Remark~\ref{rem:general}.
\end{remark}

\section{Final remarks}\label{sec:remarks}

We collect here a number of remarks on the scope of the argument: the role of the residue class of $m$ modulo $2k-1$ and its one exception, how the bundles $E_{k,m}$ for different $m$ are related and why this matters for the conjecture, what happens in order two, why there is no such obstruction on surfaces, and what is known and what is not for orders $k\ge5$.

\begin{remark}[The restriction $5\mid m$, and its one exception]\label{rem:5divides}
Let $m=5q+r$ with $1\le r\le4$ and $q\ge r$. Repeating the weight count of Step~1 in the proof of Proposition~\ref{prop:quotient}, a monomial in the generators of weight $m$ has $\deg_b\le h+d\le q$, with equality $h+d=q$ if and only if $u+3w+d=r$. Now $d>0$ is allowed: for instance $D\cdot A_{12;1}^{\,q-1}$ has weight $6+5(q-1)=5q+1$ and $\deg_b=q$.

Let $G:=\E{3}{m}/F_{q-1}$ be the extremal quotient for the $b$-filtration. Since \eqref{eq:chainrule} is also affine-linear in $a$, the degree in $a$ defines a second increasing filtration by subbundles on $G$, and the extremal quotient $G'$ of $G$ for this filtration is again a quotient bundle of $\E{3}{m}$. In every trivialisation $G'$ is spanned by the leading parts (of maximal degree in $b$, then in $a$) of the monomials with $h+d=q$. The $b$-leading part of $v^uW^wA^hD^d$ is $v^uW^wS^h\det(v,a,b)^d$, of degree $w+d$ in $a$, and maximising $w+d$ subject to $u+3w+d=r$ forces $u=w=0$ and $d=r$ (recall that $q\ge r$, so that $h=q-r\ge0$). These leading parts contain the factor $\det(v,a,b)^r$, so every irreducible summand $\Sch{\lambda}$ of $G'$ has $\lambda_3\ge r\ge1$: three rows. For example, for $r=1$ one finds
\[
G'\cong K_X\otimes\Sch{(2q-2,q-1,0)}\Om\cong\Sch{(2q-1,q,1)}\Om ,
\]
which has many sections on hypersurfaces of large degree. Theorem~\ref{thm:BR} therefore gives no information on $G'$, and the method of this note does not apply to $\E{3}{m}$ for $5\nmid m$, $q\ge r$.

For the six weights with $1\le q<r$, namely $m=7,8,9,13,14,19$, the count has to be redone. The monomials of maximal degree in $b$ and then in $a$ are
\[
\begin{array}{c|cccccc}
m&7&8&9&13&14&19\\\hline
\text{monomials}&vD&v^2D,\ WA&WD&vD^2&v^2D^2,\ WAD&vD^3
\end{array}
\]
and the conclusion is the same for $m=7,9,13,14,19$; but at $m=8$ the monomial $WA$ contains no $D$. This gives one more non-ampleness statement, Proposition~\ref{prop:eight} below. The simplest untreated case is $\E{3}{6}$, whose extremal quotient is $K_X$ itself. The same discussion applies \textsl{verbatim} in Theorem~\ref{thm:ci}, with the minors $D_{ijk}$ in place of $D$.

In order four the situation is slightly better: the generators of defect $1$ (in the sense of \S\ref{sec:k4}) have $b$-free leading parts, which is why Theorem~\ref{thm:k4} covers the residue classes $m\equiv0,1,2\pmod7$; but for $m\equiv3,\dots,6\pmod7$ the extremal quotient involves generators of defect $\ge3$ such as $W^{10}$, whose leading part $\det(v,a,b,d)$ has four rows, and the method gives no information.
\end{remark}

\begin{proposition}[The weight $8$]\label{prop:eight}
Let $Z$ be as in Theorem~\ref{thm:ci}, or $Z=X\subset\PP^4$ a smooth hypersurface of any degree. Then $\Ez{3}{8}$ admits the quotient bundle $\Sch{(3,2,0,\dots,0)}\OmZ$, and is not ample.
\end{proposition}

\begin{proof}
At weight $8=5\cdot1+3$ the monomials in the generators of Proposition~\ref{prop:alln} of maximal degree $1$ in $b$ are $v_iv_jv_kA$, $W_{ij}A$ and $v_iv_jD_{klm}$, of degree $0$, $1$, $1$ in $a$. Taking the top symbol in $b$ and then the top symbol in $a$ (Proposition~\ref{prop:symbol} applied twice, the second time to the $2$-jet differentials in which the first symbol takes values) gives a quotient bundle $G'$ of $\Ez{3}{8}$, spanned in every trivialisation by the polynomials
\[
W_{ij}(v,a)\,S_{kl;m}(v,b)\qquad\text{and}\qquad v_iv_j\det\nolimits_{klm}(v,a,b),
\]
and contained in the tensor bundle $\Sym^3\OmZ\otimes\OmZ\otimes\OmZ$ of polynomials of multidegree $(3,1,1)$ in $(v,a,b)$, since $v$, $a$ and $b$ all transform through $J$ once the lower-degree terms are discarded.

As in the proof of Corollary~\ref{cor:tensorial}, the fibre of $G'$ is one and the same $\GL_n$-stable subspace in every trivialisation, so by Lemma~\ref{lem:equivariant} every isotypic component of $G'$ is a direct summand of $G'$, hence a quotient bundle of $\Ez{3}{8}$.

Now $W_{12}S_{12;1}=W_{12}\,v_1(v_1b_2-v_2b_1)$ is a highest weight vector of weight $(3,2,0,\dots,0)$ (it is killed by the raising operators $v_j\mapsto v_i$, $a_j\mapsto a_i$, $b_j\mapsto b_i$, $i<j$, and has that weight), and $\Sch{(3,2,0,\dots,0)}$ occurs with multiplicity one in $\Sym^3V^*\otimes V^*\otimes V^*$; so $\Ez{3}{8}$ admits the quotient bundle $\Sch{(3,2,0,\dots,0)}\OmZ$. (The other irreducible summand of $G'$, the one generated by $v_1^2\det_{123}(v,a,b)$, is $\Sch{(3,1,1,0,\dots,0)}\OmZ$, which has three rows.)

A Schur bundle with two rows is not ample on $Z$, by the argument of Theorem~\ref{thm:main}: realise $\Sch{(3,2,0,\dots,0)}V^*$ as the kernel of $\delta=\sum v_i\,\partial/\partial a_i$ on the polynomials of bidegree $(3,2)$ in $(v,a)$ (Remark~\ref{rem:schur}); multiplication of polynomials maps $\Sym^r$ of it into the kernel of $\delta$ in bidegree $(3r,2r)$, which is the irreducible $\Sch{(3r,2r,0,\dots,0)}V^*$, and the image is a nonzero $\GL_n$-stable subspace (it contains the $r$-th power of the highest weight vector $v_1W_{12}^{\,2}$), hence everything:
\[
\Sym^r\Sch{(3,2,0,\dots,0)}\OmZ\twoheadrightarrow\Sch{(3r,2r,0,\dots,0)}\OmZ .
\]

Suppose now that $\Ez{3}{8}$ is ample. Then its quotient $\Sch{(3,2,0,\dots,0)}\OmZ$ is ample by (H1), and $\Sch{(3r,2r,0,\dots,0)}\OmZ$ is globally generated for some $r$ by (H2) and (H3), hence has sections; but it has two rows, so $H^0=0$ by Theorem~\ref{thm:BR} when the degrees are $\ge2$, and by the Euler sequence as in the proof of Lemma~\ref{lem:vanishing} when $X=\PP^3$.
\end{proof}

\begin{remark}
The proof shows how Theorem~\ref{thm:criterion} can be refined: the symbol can be iterated with respect to $f^{(k)},f^{(k-1)},\dots$ until the result lies in a tensor bundle, and then every isotypic component of the result is a quotient. At weight $8$ this isolates a two-row summand even though the iterated top symbol involves all three derivatives. We have not pursued this systematically; for $5\nmid m$ and $q\ge r$ the iterated top symbol of $\E{3}{m}$ is spanned by the leading parts of the monomials of the single type $D^rA^{q-r}$ (all choices of indices), which contain the factor $\det(v,a,b)^r$, so that every irreducible summand has three rows (Remark~\ref{rem:5divides}) and no further weights arise in order $3$ in this way.
\end{remark}

\begin{remark}[How the bundles $E_{k,m}$ are related for varying $m$]\label{rem:varym}
This point is rarely made explicit in the literature, so we take the opportunity to spell it out here. The natural $\cO_Z$-linear relation between $E_{k,m}T^*_Z$ and $E_{k,ml}T^*_Z$ is the multiplication map
\begin{equation}\label{eq:multmap}
\mu_{m,l}\colon\ \Sym^lE_{k,m}T^*_Z\longrightarrow E_{k,ml}T^*_Z ,
\end{equation}
a morphism of vector bundles compatible with the transition functions \eqref{eq:faadibruno}. For $k=1$ one has $E_{1,m}=\Sym^m\Omega^1_Z$, the algebra $\bigoplus_mE_{1,m}$ is the symmetric algebra of $\Omega^1_Z$, and $\mu_{m,l}$ is surjective for all $m,l$; together with Hartshorne's theorem that $E$ is ample if and only if $\Sym^mE$ is, this is why ampleness of $E_{1,m}$ does not depend on $m$.

For $k\ge2$ the algebra $\bigoplus_mE_{k,m}$ is not generated in weight $1$, nor in any single weight in general, and $\mu_{m,l}$ need not be surjective: for instance $\mu_{1,l}$ has image $\Sym^l\Omega^1_Z\subsetneq E_{k,l}$ as soon as $l\ge3$ (it misses the $W_{ij}$), and for $k=3$ the image of $\mu_{3,2}\colon\Sym^2E_{3,3}\to E_{3,6}$ consists of polynomials of degree $0$ in $f'''$, since $E_{3,3}$ is spanned by the $v_iv_jv_k$ and the $W_{ij}$, so it misses $D$. (Non-injectivity of $\mu_{m,l}$ is not the relevant issue: $\Sym^l\Sym^m\Omega^1_Z\to\Sym^{ml}\Omega^1_Z$ already has a kernel for $m\ge2$.)

Consequently, when $\mu_{m,l}$ is not surjective, ampleness of $E_{k,m}$ gives only an ample \emph{subsheaf} of $E_{k,ml}$, from which nothing follows for $E_{k,ml}$ itself, and conversely ampleness of $E_{k,ml}$ says nothing about $E_{k,m}$: no implication in either direction is available. Since Theorem~\ref{thm:main} settles $E_{3,5q}$ but leaves $E_{3,6}$ open (Remark~\ref{rem:5divides}), our results are compatible both with ampleness actually depending on $m$ and with non-ampleness at every weight.

The positivity which \emph{is} independent of $m$ lives on the Demailly--Semple tower, where $E_{k,m}T^*_Z=\pi_{k,0\,*}\cO_{Z_k}(m)$ and $\cO_{Z_k}(ml)=\cO_{Z_k}(m)^{\otimes l}$: bigness, the stable base locus, the augmented base locus and the notion of almost $k$-jet ampleness of \cite{Deng20} are properties of the line bundle $\cO_{Z_k}(1)$ and are insensitive to replacing $m$ by a multiple.

This is also how Demailly proves that ampleness of $E_{k,m}T^*_X$ implies hyperbolicity \cite[7.7~iii) and Cor.~7.10]{Dem97}: the sections of $\Sym^pE_{k,m}T^*_X\otimes L^{-1}$ provided by Hartshorne's theorem are pushed through $\mu_{m,p}$ to sections of $\cO_{X_k}(mp)\otimes\pi^*L^{-1}$ whose base locus is $X_k^{\rm sing}$: indeed, since $\Sym^pE_{k,m}T^*_X\otimes L^{-1}$ is globally generated, on each fibre $X_{k,x}$ of $X_k\to X$ these sections span the products of $p$ elements of $E_{k,m,x}$, whose common zeros are those of $E_{k,m,x}$ itself, namely $X_{k,x}^{\rm sing}$, because the relative base locus of $|\cO_{X_k}(m)|$ over $X$ is $X_k^{\rm sing}$ \cite[Thm.~6.8~iii)]{Dem97}. From that point on the weight plays no role.

Formulations of Conjecture~\ref{conj:DT} with ``for every $m$'' or ``$m\gg0$'' (\S\ref{sec:conjecture}) are natural for this line-bundle positivity, and for $k=1$ they are harmless, since $\Sym^m\OmZ$ is ample for one $m$ if and only if it is for all; they become delicate exactly when transferred to the vector bundles $E_{k,m}T^*_Z$ with $k\ge2$, for which the obstruction of this note depends on $m$ modulo $2k-1$ (whether ampleness itself does, our results do not decide).

There is, however, one general positive statement. Since the fibre algebra $\bigoplus_mE_{k,m}$ is the same finitely generated graded algebra at every point whenever finitely many generators are known (for $k\le3$ in every dimension by Proposition~\ref{prop:alln}, and in the cases computed in \cite{Mer10}), a standard pigeonhole argument on Veronese subrings shows that if the generators have weights $d_1,\dots,d_s$, then for every $m$ divisible by $s\cdot\operatorname{lcm}(d_1,\dots,d_s)$ every monomial of weight $ml$ in the generators is a product of $l$ monomials of weight $m$, so that $\mu_{m,l}$ is surjective for all $l\ge1$. For such $m$, $E_{k,ml}T^*_Z$ is a quotient of $\Sym^lE_{k,m}T^*_Z$, and ampleness of $E_{k,m}T^*_Z$ implies ampleness of $E_{k,ml}T^*_Z$ for every $l$ by (H1).

Note however that such an $m$ is already a multiple of $5$ in order $3$ and of $7$ in order $4$, so this adds nothing to Theorems~\ref{thm:ci} and~\ref{thm:k4}; the real content of the multiplication map in our situation is the following.
\end{remark}

\begin{corollary}\label{cor:divisible}
Let $Z$ be as in Theorem~\ref{thm:ci} (or $Z=X\subset\PP^4$ as in Theorem~\ref{thm:main}) and let $m\ge1$. If $\Ez{3}{m}$ is ample, then the multiplication map $\mu_{m,5}\colon\Sym^5\Ez{3}{m}\to\Ez{3}{5m}$ is not surjective. Likewise, if $Z$ is as in Theorem~\ref{thm:k4} and $\Ez{4}{m}$ is ample, then $\mu_{m,7}$ is not surjective. In particular reading \textup{(B)} of Conjectures~\ref{conj:DT} and~\ref{conj:Deng} can hold at $k=3$ (resp.\ $k=4$), in the range of Theorem~\ref{thm:ci} (resp.\ \ref{thm:k4}), only at weights $m$ for which $\mu_{m,5}$ (resp.\ $\mu_{m,7}$) fails to be surjective.
\end{corollary}

\begin{proof}
If $\mu_{m,5}$ were surjective, the bundle $\Ez{3}{5m}$ would be a quotient of $\Sym^5\Ez{3}{m}$, hence ample by (H1), which contradicts Theorem~\ref{thm:ci} (Theorem~\ref{thm:main} for $X\subset\PP^4$). The same argument applies to $\mu_{m,7}$, $\Ez{4}{7m}$ and Theorem~\ref{thm:k4}.
\end{proof}

\begin{remark}[Order two, and consistency with the threshold]\label{rem:ordertwo}
The argument of Theorem~\ref{thm:criterion}, extremal quotient plus Br\"uckmann--Rackwitz vanishing, applies also to order $2$, in every dimension and, there, at every weight. Since the invariant $2$-jet algebra is generated by the $v_i$ and the $W_{ij}$ (see \S\ref{sec:general}), for $m\ge3$ and $w=\lfloor m/3\rfloor$ the extremal quotient for the $a$-filtration of $\Ez{2}{m}$ is spanned by the products of $w$ symbols $v_ia_j-v_ja_i$ with $m-3w$ factors $v_i$, hence is a subbundle of $\Sym^{m-2w}\OmZ\otimes\Sym^w\OmZ$ (two rows), isomorphic to $\Sch{(w,w,0,\dots,0)}\OmZ$ when $m=3w$. By Theorem~\ref{thm:BR}, $\Ez{2}{m}$ is not ample on any smooth complete intersection with $3c<N$, for every $m\ge1$ (for $m\le2$, $\Ez{2}{m}=\Sym^m\OmZ$).

This is exactly complementary to the range $N\le3c$ in which Conjecture~\ref{conj:DT} predicts ampleness for $k=2$, so at order $2$ the two-row obstruction is sharp with respect to the threshold, and there is no contradiction: it is the case $k=2$ of Remark~\ref{rem:below}. For $k=1$ the argument reduces to the classical fact that $\Sym^m\OmZ$ is not ample when $2c<N$, since it has no sections.
\end{remark}

\begin{remark}[Dimension two]\label{rem:dim2}
No obstruction of this kind exists on surfaces. For a surface in $\PP^3$ Theorem~\ref{thm:BR} requires $\lambda'_1<2$, \textsl{i.e.}\ one row (and for complete intersection surfaces of higher codimension even more), whereas the extremal quotients produced above have two rows, and on a surface every irreducible summand of any graded piece of a jet bundle has at most two rows.

Concretely, for a smooth surface $X\subset\PP^3$ of degree $d$ the extremal quotients are the following.
\begin{itemize}
\item For $\E{2}{3w}$: $\Sch{(w,w)}\Om=K_X^{\,w}$, ample for $d\ge5$.
\item For $\E{3}{5q}$: $\Sch{(2q,q)}\Om=\Sym^q(\Om\otimes K_X)$, where
\[
\Om\otimes K_X=\Om(2)\otimes\cO_X(d-6)
\]
is ample for $d\ge7$, because $\Om(2)$ is globally generated, being a quotient of $\Omega^1_{\PP^3}(2)|_X$, and it is well known that $\Omega^1_{\PP^N}(2)$ is globally generated.
\item For $\E{4}{7q}$: since among the generators of the invariant $4$-jet algebra in dimension $2$ \cite[\S1]{Mer10}, \cite{Mer08} only the $\Lambda^7_{i,j}$ reach the ratio $1/7$, by the argument of Remark~\ref{rem:schur} the extremal quotient is
\[
\Sch{(3q,q)}\Om=\Sym^{2q}\Om\otimes K_X^{\,q}=\Sym^{2q}(\Om(2))\otimes\cO_X(q(d-8)),
\]
ample for $d\ge9$.
\end{itemize}
(For $d=7,8$ this last bundle is not ample on every smooth surface: if $X$ contains a line $L$, the quotient $\Om|_L\twoheadrightarrow\Omega^1_L=\cO_L(-2)$ induces a quotient $\cO_L(q(d-8))$ of its restriction to $L$, of degree $\le0$; and smooth surfaces containing lines exist in every degree, for instance the Fermat surface of degree $d$, which contains $3d^2$ of them.)

None of this affects the point of the remark: no obstruction to ampleness of jet bundles on surfaces comes from the extremal quotients, consistently with the threshold $k\ge2$ of Conjecture~\ref{conj:DT} for surfaces in $\PP^3$, about which the present method says nothing.
\end{remark}

\begin{remark}[Orders $k\ge5$]\label{rem:general}
Two of the ingredients used above are available in general, and one is not.
\begin{enumerate}
\item \emph{Invariants of weight $2k-1$, linear in $f^{(k)}$, exist for all $k$ and $n$.} Demailly observes in the proof of \cite[Thm.~6.8~i)]{Dem97} that $(f_r')^{2k-1}g_i^{(k)}$ is an invariant polynomial of weight $2k-1$, where $g=f\circ\varphi$ denotes the reparametrisation with $g_r(\tau)=\tau$, and by \cite[(6.9)]{Dem97} its leading term in $f^{(k)}$ is $(f_r')^{k-2}\big(f_i^{(k)}f_r'-f_r^{(k)}f_i'\big)$. Merker constructs the same invariants $\Lambda^{2k-1}_{1,i;1^{k-2}}$ by iterated bracketing with $f_1'$ \cite[\S5, ``Arbitrary dimension'']{Mer10}. Products of $q$ of their leading terms and of their $\GL_n$-polarisations span the irreducible two-row summand
\[
\Sch{((k-1)q,q,0,\dots,0)}\OmZ\subset\Sym^{(k-1)q}\OmZ\otimes\Sym^q\OmZ ,
\]
by the argument of Remark~\ref{rem:schur} (the polarisations are needed: for $k=3$, $n=3$, $q=1$ the displayed terms alone span a space of dimension $6<8=\dim\Sch{(2,1,0)}$). In particular, whenever the invariants of weight $2k-1$ which are linear in $f^{(k)}$ are known to be spanned by the polarisations of $\Lambda^{2k-1}$ (as is the case for $k\le4$ by \S\S\ref{sec:general}--\ref{sec:k4}), the bundle $E_{k,2k-1}T^*_Z$ has the two-row quotient $\Sch{(k-1,1,0,\dots,0)}\OmZ$, since $\deg_{f^{(k)}}\le1$ at weight $2k-1$ for weight reasons, and is not ample when $3c<N$ (Theorem~\ref{thm:criterion} with $s=1$).
\item \emph{The degree bound $\deg_{f^{(k)}}P\le\operatorname{wt}(P)/(2k-1)$ holds in every case where generators are known}: for $k=2,3$ in all dimensions (\S\ref{sec:general}), for $k=4$ in all dimensions $n\ge4$ (\S\ref{sec:k4}) and $n\le3$ (by restriction), and for $(n,k)=(2,5)$: among Merker's $17$ generating bi-invariants \cite[\S10]{Mer10} (see also \cite{Mer08}), the five of weight $\le8$ are order-four invariants and do not involve $f^{(5)}$ at all, $\Lambda^9$ is linear in $f^{(5)}$, and the degree in $f^{(5)}$ of each remaining generator is read off from its defining syzygy; for instance, the syzygies
\[
f_1'X^{18}\equiv-5\Lambda^9M^{10}+56\Lambda^7K^{12},\qquad f_1'Y^{27}\equiv M^{10}X^{18}-56K^{12}F^{16}
\]
give degree $\le2$ for $X^{18}$ and degree $\le3$ for $Y^{27}$; in every case one finds $\deg_{f^{(5)}}\le\operatorname{wt}/9$, with equality exactly for $\Lambda^9$, $X^{18}$ and $Y^{27}$.
\item \emph{What is missing for $k\ge5$ and $n\ge3$.} No generating set of the invariant algebra is available, so neither the degree bound nor the structure of the leading parts of the extremal generators is known. The dimension-$2$ computation for $k=5$ shows that the extremal ratio $1/9$ is attained, besides $\Lambda^9$, by $X^{18}$ and $Y^{27}$; by analogy with $k=4$ one expects their leading parts to involve only $f',f'',f'''$ besides $f^{(5)}$ (Remark~\ref{rem:whyrows}), so that Theorem~\ref{thm:criterion} would apply with $s=3$, which would give, for hypersurfaces in $\PP^6$, an extremal quotient of $E_{5,9q}$ with at most four rows and hence a contradiction with reading (A) at $k=5$, $N=6$, $c=1$. We have no proof of this.
\end{enumerate}
\end{remark}

\end{document}